\documentclass[a4paper,12pt]{amsart}
\usepackage{mathrsfs}
\usepackage{appendix}
\usepackage{amssymb}
\usepackage{fancyhdr}
\usepackage{charter}
\usepackage{typearea}
\usepackage{color}
\usepackage{amsmath,amstext,amsthm,amscd}

\usepackage[a4paper,top=2.3cm,bottom=2.3cm,left=2cm,right=2cm]{geometry}

\makeatletter
      \def\@setcopyright{}
      \def\serieslogo@{}
      \makeatother

\newcommand{\Complex}{\mathbb C}
\newcommand{\Real}{\mathbb R}
\newcommand{\N}{\mathbb N}
\newcommand{\ddbar}{\overline\partial}
\newcommand{\pr}{\partial}
\newcommand{\ol}{\overline}
\newcommand{\Td}{\widetilde}
\newcommand{\norm}[1]{\left\Vert#1\right\Vert}
\newcommand{\abs}[1]{\left\vert#1\right\vert}

\newcommand{\set}[1]{\left\{#1\right\}}
\newcommand{\To}{\rightarrow}

\theoremstyle{plain}
\newtheorem{thm}{Theorem}[section]
\newtheorem{cor}[thm]{Corollary}

\newtheorem{lem}[thm]{Lemma}
\newtheorem{prop}[thm]{Proposition}

\theoremstyle{definition}
\newtheorem{defn}[thm]{Definition}

\theoremstyle{remark}
\newtheorem{rem}[thm]{Remark}

\numberwithin{equation}{section}

\begin{document}
\title[Szeg\H{o} kernel asymptotic expansion on strongly pseudoconvex CR manifolds with $S^1$ action]
{Szeg\H{o} kernel asymptotic expansion on strongly pseudoconvex CR manifolds with $S^1$ action}
\author[Hendrik Herrmann]{Hendrik Herrmann}
\address{Mathematical Institute, University of Cologne, Weyertal 86-90, 50931 Cologne, Germany}
\thanks{The first named author would like to thank the Institute of Mathematics, Academia Sinica for hospitality, a comfortable accommodation and financial support during his visits in August 2016 and January 2017}
\email{heherrma@math.uni-koeln.de or post@hendrik-herrmann.de}
\author[Chin-Yu Hsiao]{Chin-Yu Hsiao}
\address{Institute of Mathematics, Academia Sinica and National Center for Theoretical Sciences, Astronomy-Mathematics Building, No. 1, Sec. 4, Roosevelt Road, Taipei 10617, Taiwan}
\thanks{The second named author was partially supported by Taiwan Ministry of Science of Technology project 104-2628-M-001-003-MY2 and the Golden-Jade fellowship of Kenda Foundation}
\email{chsiao@math.sinica.edu.tw or chinyu.hsiao@gmail.com}
\author[Xiaoshan Li]{Xiaoshan Li}
\address{School of Mathematics
and Statistics, Wuhan University, Hubei 430072, China}
\thanks{The third  author was  supported by  NSFC No. 11501422 and he also thank the Institute of Mathematics, Academia Sinica for hospitality during his visit in January 2017.}
%\thanks{The second author was  supported by Central university research Fund 2042015kf0049, Postdoctoral Science Foundation of China 2015M570660 and NSFC No. 11501422}
%\thanks{The second-named author was partially supported by Central university research Fund 2042015kf0049.}
\email{xiaoshanli@whu.edu.cn}
%\author[Hendrik Herrmann]{Hendrik Herrmann}
%\address{School of Mathematics
%and Statistics, Wuhan University, Hubei 430072, China}
%\dedicatory{Dedicated to Professor Ngaiming Mok for his 60th birthday}

\begin{abstract}
Let $X$ be a compact connected strongly pseudoconvex CR manifold of dimension $2n+1, n \ge 1$ with a transversal CR $S^1$ action on $X$. We establish an asymptotic expansion for the $m$-th Fourier component of the Szeg\H{o} kernel function as $m\To\infty$, where the expansion involves a contribution in terms of a distance function from lower dimensional strata
of the $S^1$~action. We also obtain explicit formulas for the first three coefficients of the expansion.
\end{abstract}

\maketitle \tableofcontents

\section{Introduction and statement of the main results} \label{s-gue160527}

Let $L$ be a holomorphic line bundle over a complex manifold $M$ and let $L^k$ be the $k$-th tensor power of $L$.
The Bergman kernel is the smooth kernel of the orthogonal projection onto the space of $L^2$-integrable holomorphic sections of $L^k$.
The study of the large $k$ behaviour of the Bergman kernel is an important research subject in complex geometry.
%and is closely related to topics like the the existence of canonical K\"ahler metrics , Berezin-Toeplitz quantization and  equidistribution of zeros of holomorphic sections and mathematical physics.
In the case of a positive line bundle $L$ over a compact base manifold $M$, Tian~\cite{T90} obtained results on the asymptotic behavior of the Bergman kernel by using peak section methods.
Catlin~\cite{Cat97} and Zelditch~\cite{Zel98}
established the asymptotic expansion of the Bergman kernel by using a fundamental result by Boutet de Monvel-Sj\"{o}strand~\cite{BouSj76} about the asymptotics of the Szeg\H{o} kernel on a strongly pseudoconvex boundary.
% {\color{red}which was generalized by Hsiao\cite{H08} to the Szeg\H{o} kernel of Szeg\H{o} projection for $(0, q)$-forms when the Levi-form is non-degenerate. Recently, in \cite{HHL17a} the same authors of this paper generalized the Boutet de Monvel-Sjostrand's result to the Heisenberg group which is  a model of non-compact CR manifold with non-degenerate Levi form.}
Another proof of the existence of the full asymptotic expansion for the Bergman kernel was obtained by Berman, Berndtsson and Sj\"ostrand~\cite{BBS04}. Dai, Liu and Ma~\cite{DLM04a, DLM06} obtained the full off-diagonal asymptotic expansion and Agmon estimates of the Bergman kernel by using the heat kernel method. Their result holds actually for the more general
Bergman kernel of the spin$^{c}$ Dirac operator associated to a positive line bundle on a
compact symplectic manifold. The Bergman kernel asymptotic expansion plays an important role in many recent research topics, for example, the the existence of canonical K\"ahler metrics , Berezin-Toeplitz quantization, equidistribution of zeros of holomorphic sections and mathematical physics.
We refer the reader to the book \cite{MM07}
for a comprehensive study of the Bergman kernel and its applications and also to the survey \cite{Ma10}.

In CR geometry, among those transversally elliptic operators initiated by Atiyah and Singer, Kohn's $\Box_b$ operator on CR manifolds with $S^1$ action ($\mathbb R$ action, torus action) is a natural one of geometric significance for complex analysts.
Recently, we established Morse inequalities (\cite{HsiaoLi15}, \cite{HL1}), a local index theorem (\cite{CHT}) and Kodaira embedding theorem (\cite{HHL15}, \cite{Hsiao14}, \cite{HLM16}, \cite{HHL17b}) on such manifolds.
As in complex geometry, to study further fundamental CR geometric problems (for example, CR Donaldson's program~\cite{Do:01}, geometric quantization of CR manifolds), it is crucial to be able to know the asymptotic behaviour of $m$-th Fourier component of the Szeg\H{o} kernel function. This is the motivation of this work. In this paper, we establish an asymptotic expansion for $m$-th Fourier component of the Szeg\H{o} kernel function as $m\To+\infty$, where the expansion involves a contribution in terms of a distance function from lower dimensional strata
of the $S^1$ action. We also obtain explicit formulas for the first three coefficients of the expansion.
It should be mentioned that in~\cite{HLM16} ($S^1$-action case) and ~\cite{HHL17b} ($\Real$-action case), we consider a positive CR line bundle $L$ over a compact not necessary strongly pseudoconvex CR manifold $X$ and we sum over the Fourier component of the Szeg\H{o} kernel function until some $k\geq1$  with values in $L^k$ and we establish an asymptotic expansion in $k$ for this object.  In this work, we replace  the existence of a positive CR line bundle by the strongly pseudoconvexity condition and we can establish an asymptotic expansion for each Fourier component of the Szeg\H{o} kernel function.

We now formulate the main results. We refer to Section~\ref{s:prelim} for some notations and terminology used here.
Let $(X, T^{1,0}X)$ be a compact connected strongly pseudoconvex CR manifold with a transversal CR locally free $S^1$ action $e^{i\theta}$ (see Definition~\ref{d-gue160502}), where $T^{1,0}X$ is a CR structure of $X$.
Let $T\in C^\infty(X,TX)$ be the real vector field induced by the $S^1$ action and let $\omega_0\in C^\infty(X,T^*X)$ be the global real one form determined by
\begin{equation}\label{e-gue160502b}
\langle\,\omega_0\,,\,T\,\rangle=-1,\ \ \langle\,\omega_0\,,\,u\,\rangle=0,\ \ \forall u\in T^{1,0}X\oplus T^{0,1}X.
\end{equation}
Take a rigid Hermitian metric $\langle\,\cdot\,|\,\cdot\,\rangle$ on $\Complex TX$ such that $T^{1,0}X\perp T^{0,1}X$, $T\perp (T^{1,0}X\oplus T^{0,1}X)$, $\langle\,T\,|\,T\,\rangle=1$ holds (see Definition~\ref{d-gue150514f} and the discussion after Definition~\ref{d-gue150514f}) and let $dv_X=dv_X(x)$ be the volume form on $X$ induced by the rigid Hermitian metric $\langle\,\cdot\,|\,\cdot\,\rangle$ on $\Complex TX$. Then we get a natural global $L^2$ inner product $(\,\cdot\,|\,\cdot\,)$ on $C^\infty(X)$. We denote by $L^{2}(X)$ the completion of $C^\infty(X)$ with respect to $(\,\cdot\,|\,\cdot\,)$.

Let $\ddbar_b:\Omega^{0,q}(X)\To\Omega^{0,q+1}(X)$ be the tangential Cauchy-Riemann operator, $q=0,1,2,\ldots,n$. For every $u\in\Omega^{0,q}(X)$, we can define $Tu:=\mathscr L_Tu\in\Omega^{0,q}(X)$ (see \eqref{e-gue150508faII} and \eqref{lI}), where $\mathscr L_Tu$ denotes the Lie derivative
of $u$ along the direction $T$.
We can check that $T\ddbar_b=\ddbar_bT$ on $\Omega^{0,q}(X)$, $q=0,1,2,\ldots,n$. For $m\in\mathbb Z$, put
\[C^\infty_m(X):=\set{u\in C^\infty(X);\, Tu=imu}.\]
For each $m\in\mathbb Z$, set
\begin{equation}\label{e-gue160604}
H^0_{b,m}(X):=\set{u\in C^\infty_m(X);\, \ddbar_bu=0}.
\end{equation}
It is well-known that (see Theorem 3.5 in~\cite{CHT}) ${\rm dim\,}H^0_{b,m}(X)<\infty$. Let $f^m_1\in H^0_{b,m}(X),\ldots,f^m_{d_m}\in H^0_{b,m}(X)$ be an orthonormal basis for $H^0_{b,m}(X)$ with respect to $(\,\cdot\,|\,\cdot\,)$. The $m$-th Fourier component of the Szeg\H{o} kernel function is given by
\begin{equation}\label{e-gue160605}
S_m(x):=\sum^{d_m}_{j=1}\abs{f^m_j(x)}^2\in C^\infty(X).
\end{equation}

For $x\in X$, we say that the period of $x$ is $\frac{2\pi}{q}$, $q\in\mathbb N$, if $e^{i\theta}\circ x\neq x$ for every $0<\theta<\frac{2\pi}{q}$ and $e^{i\frac{2\pi}{q}}\circ x=x$. For each $q\in\mathbb N$, put
\begin{equation}\label{e-gue150802bm}
X_q=\set{x\in X;\, \mbox{the period of $x$ is $\frac{2\pi}{q}$}}
\end{equation}
and set $p=\min\set{q\in\mathbb N;\, X_q\neq\emptyset}$. It is well-known that if $X$ is connected, then $X_p$ is an open and dense subset of $X$ (see Duistermaat-Heckman~\cite{Du82}). We denote $X_{{\rm reg\,}}:=X_{p}$. We call $x\in X_{{\rm reg\,}}$ a regular point of the $S^1$ action. Let $X_{{\rm sing\,}}$ be the complement of $X_{{\rm reg\,}}$. Assume $X=X_{p_1}\bigcup X_{p_2}\bigcup\cdots\bigcup X_{p_t}$, $p=:p_1<p_2<\cdots<p_t$. Put $X_{{\rm sing\,}}=X^1_{{\rm sing\,}}:=\bigcup^t_{j=2}X_{p_j}$, $X^{r}_{{\rm sing\,}}:=\bigcup^t_{j=r+1}X_{p_j}$, $t-1\geq r\geq 2$. Set $X^{t}_{{\rm sing\,}}:=\emptyset$. For $x, y\in X$, let $d(x,y)$ denotes the standard Riemannian distance of $x$ and $y$ with respect to the given Hermitian metric. Let $A$ be any subset of $X$. For $x\in X$, put $d(x,A):=\inf\set{d(x,y);\, y\in A}$ if $A\neq\emptyset$, $d(x,A)=0$ if $A=\emptyset$. The main result of this work is the following

\begin{thm}\label{t-gue160605}
Let $X$ be a compact connected strongly pseudoconvex CR manifold of dimension $2n+1$, $n\geq1$, with a transversal CR $S^1$ action $e^{i\theta}$. With the notations used above, there are $b_j(x)\in C^\infty(X)$, $j=0, 1,2,\ldots$,  and $\varepsilon_0>0$ such that for any $r=1,\ldots,t$, any differential operator $P_\ell:C^\infty(X)\To C^\infty(X)$
of order $\ell\in\mathbb N_0$ and every $N\in\mathbb N$,  there is a $C_{N,P_\ell}>0$ such that
\begin{equation}\label{e-gue160605a}
\begin{split}
\abs{P_\ell\Bigr(S_m(x)-\sum\limits^{p_r}_{s=1}e^{\frac{2\pi(s-1)}{p_r}mi}\sum^{N-1}_{j=0}m^{n-j}b_{j}(x)\Bigr)}&\\
\leq C_{N,P_\ell}\Bigr(m^{n-N}+m^{n+\frac{\ell}{2}}&e^{-m\varepsilon_0d(x,X^{r}_{{\rm sing\,}})^2}\Bigr),\ \ \forall m\geq 1,\  \ \forall x\in X_{p_r}.
\end{split}
\end{equation}
\end{thm}

\begin{rem}\label{r-gue160605}
When $m$ is a multiple of $p_r$, the number $\sum\limits^{p_r}_{s=1}e^{\frac{2\pi(s-1)}{p_r}mi}$ is equal to $p_r$. When $m$ is not a multiple of $p_r$, the number $\sum\limits^{p_r}_{s=1}e^{\frac{2\pi(s-1)}{p_r}mi}$ is equal to $0$.
\end{rem}
\begin{cor}\label{c-gue160605I}
With the notations and assumptions used in Theorem~\ref{t-gue160605}, assume that $p_1=1$. For every $N\in\mathbb N$,  there are $\varepsilon_0>0$ and $C_N>0$ such that
\begin{equation}\label{e-gue160605ab}
\abs{S_m(x)-\sum^{N-1}_{j=0}m^{n-j}b_{j}(x)}\leq C_N\Bigr(m^{n-N}+m^{n}e^{-m\varepsilon_0d(x,X_{{\rm sing\,}})^2}\Bigr),\ \ \forall m\geq1,\  \ \forall x\in X_{{\rm reg\,}}.
\end{equation}
\end{cor}

%\begin{rem}\label{r-gue160605I}
%We actually prove that (see Theorem~\ref{t-gue160905}) there are $m$-dependent functions $b_j(m,x)\in C^\infty(X)$ with $\norm{b_j(m,x)}_{C^\ell(X)}\leq C_\ell m^\ell$ on $X$, for all $m\geq1$,  $\ell\in\mathbb N_0$, where $C_\ell>0$ is a constant independent of $m$, $j=0, 1,2,\ldots$, such that
%\begin{equation}\label{e-gue160605aI}
%S_m(x)\sim m^nb_0(m,x)+m^{n-1}b_1(m,x)+\cdots\ \ \mbox{as $m\To+\infty$}.
%\end{equation}
%(See Definition~\ref{d-gue150608}.)

%Moreover, for any $r=1,\ldots,t$, any differential operator $P_\ell:C^\infty(X)\To C^\infty(X)$
%of order $\ell\in\mathbb N_0$, every $j=0,1,2,\ldots$,  there are $\varepsilon_0>0$ and $C_0>0$ such that
%\begin{equation}\label{e-gue160605aII}
%\abs{P_\ell\Bigr(b_j(m,x)-\sum\limits^{p_r}_{s=1}e^{\frac{2\pi(s-1)}{p_r}mi}b_{j}(x)\Bigr)}\leq C_0m^{\frac{\ell}{2}}e^{-m\varepsilon_0 d(x,X^{r}_{{\rm sing\,}})^2},\ \forall m\geq 1\  \ \forall x\in X_{p_r},
%\end{equation}
%where $b_j(x)\in C^\infty(X)$, $j=0,1,2,\ldots$, are as in Theorem~\ref{t-gue160605} (see Theorem~\ref{t-gue160905c}).
%\end{rem}

\begin{rem}[Relations to Bergman kernels on orbifolds]\label{Rem1}
	It is well-known that a compact connected strongly pseudoconvex CR manifold \(X\) with a transversal CR $S^1$ action can be identified with the circle bundle of a positive holomorphic orbifold line bundle \((L,h)\rightarrow M\) over a compact complex orbifold \(M\) such that a rigid Hermitian metric on \(X\) induces a Hermitian metric on \(M\) (see \cite{OV07}, \cite{BG00}, \cite{Mo88}). By this identification the $m$-th Fourier component of the Szeg\H{o} kernel function and the Bergman kernel function for the \(m\)-th tensor power of \(L\) are equal up to a factor of \(2\pi\). Hence the expansion result for Bergman kernels on orbifolds by Ma--Marinescu (\cite{MM07}, Theorems~5.4.10, 5.4.11, or Dai--Liu--Ma  \cite{DLM06}, Theorem 1.4 and (5.25) for a more general version) implies an expansion for the Fourier components of the Szeg\H{o} kernel on \(X\). However, comparing the expansion coming from their result with the expansion given in Theorem~\ref{t-gue160605}, some differences appear. For example, %Theorem~\ref{t-gue160605} involves an expansion also on the singular part of \(X\) with an error term depending on the distance from lower dimensional strata.
	%Furthermore,
	the error term in Theorem~\ref{t-gue160605} decreases faster than the error term in \cite{MM07} when \(m\rm{d}(x,X_{\rm{sing}})^2\) goes to infinity.
	
	It should be mentioned that the results by Dai--Liu--Ma \cite{DLM04a, DLM06, DLM12} and Ma--Marinescu \cite{MM07} work for general complex orbifolds while the setting considered in this work just implies results for orbifolds with cyclic quotient singularities.
	However, that kind of setting itself is interesting. Bergman kernel expansion in that specific case  was also considered by Ross--Thomas (\cite{RT11}, Theorem~1.7). They studied an object called weighted Bergman kernel which is a finite weighted sum of Bergman kernels for different tensor powers of the orbifold line bundle \(L\). They proved that this object admits a global asymptotic expansion on \(M\) in some \(C^r\)-norm up to some order \(N\) where \(r,N\in\N\) depend on the choice of weights. It was shown by Dai--Liu--Ma \cite{DLM12} that this result can be directly deduced from their work \cite{DLM06}. %Comparing these results  for orbifolds with cyclic quotient singularities to our result, the main difference is the following: While they consider a weighted sum to obtain an object with a global asymptotic expansion, we are considering each term in this sum separately and unweighted, in other words we focus on the usual Bergman kernel for this setting. Here, the difficulty is that the expansion of this object cannot be smooth in a neighborhood of the singular part of the orbifold. In particular, we show that the expansion is even discontinuous.  We then study the asymptotic behavior of the Bergman kernel for points in the singular set and prove an asymptotic expansion result. It turns out that the restriction of the Bergman kernel to certain subsets of the singular set has a smooth asymptotic expansion up to arbitrary order and we calculate the first three coefficients of this expansion in terms of local geometric data and the type of singularity. Furthermore, we describe the discontinuous  behavior  of the expansion between different classes of points in the singular set.}
	
	%On the other hand, if the term \(m\rm{d}(x,X_{\rm{sing}})\) can be bounded by a constant we have that the error term in the estimate for the \(C^\ell\) norm, \(\ell>0\), in \cite{MM07} increases slower (\(\sim m^{n+\ell/2}\)) than the error term in the main result of this work (\(\sim m^{n+\ell}\)).
\end{rem}

\begin{rem}
%\color{red}{In \cite{HHL15} we proved an embedding result for strongly pseudoconvex CR manifolds with transversal \(S^1\)-action using the asymptotics of Szeg\H{o} kernels for positive Fourier components. For this reason we also studied the asymptotic behavior of the Szeg\H{o} kernel \(S_m\) when \(m\) becomes large.  The main difference with respect to the expansion on the irregular part of \(X\) compared to this work is the following:  In \cite{HHL15} we fixed a point \(x_0\in X_k\), \(k>1\), and described the behavior of \(S_m(x,x_0)\) for \(x\) in an open neighborhood of \(x_0\) in \(X\) when \(m\) goes to infinity. As a consequence we obtained a (point-wise)  expansion for \(S_m(x_0,x_0)\). But we could not say anything about the asymptotic behavior of \(S_m(x,x)\) and its derivatives in \(x\) when \(m\) becomes large and \(x\) goes to \(x_0\). In this work we formulate and prove a result which covers that case (see Theorem~\ref{t-gue160605}) and in addition calculate the first three coefficients in the asymptotic expansion of \(S_m(x,x)\) (see Theorem~\ref{t-gue160605I}). Furthermore, the results in both works should be treated independently from each other since the methods of proof are  different. In \cite{HHL15} the starting point is the famous result on general Szeg\H{o} kernels for strictly pseudoconvex CR manifolds by Boutet de Monvel and Sj\"orstrand \cite{BouSj76}, while in this work we use a deep result on Bergman kernel expansion by Hsiao--Marinescu \cite{HM12} to obtain our results. }
In \cite{HHL15} we proved an embedding result for strongly pseudoconvex CR manifolds with transversal \(S^1\)-action using the asymptotics of Szeg\H{o} kernels for positive Fourier components. For this reason we also studied the asymptotic behavior of the Szeg\H{o} kernel \(S_m\) when \(m\) becomes large.  The main difference with respect to the expansion on the irregular part of \(X\) compared to this work is the following: %In \cite{HHL15} we fixed a point \(x_0\in X_k\), \(k>1\), and described the behavior of \(S_m(x,x_0)\) for \(x\) in an open neighborhood of \(x_0\) in \(X\) when \(m\) goes to infinity. As a consequence we obtained a (point-wise)  expansion for \(S_m(x_0,x_0)\). But we could not say anything about the asymptotic behavior of \(S_m(x,x)\) and its derivatives in \(x\) when \(m\) becomes large and \(x\) goes to \(x_0\).}
In \cite{HHL15} we fixed a point \(x_0\in X_k\), \(k>1\), then by the classical result of Boutet de Monvel-Sj\"{o}strand~\cite{BouSj76}, we have
\begin{equation}\label{e-gue180510}
S_m(x_0,x_0)=\int^{\pi}_{-\pi}\int^\infty_0e^{i\varphi(x_0,e^{i\theta}x_0)}a(x_0,e^{i\theta}x_0,t)dtd\theta,
\end{equation}
where $\varphi$ is a complex phase function and $a(x,y,t)$ is a
classical symbol of type $(1,0)$ and order $n$. 	 By using \eqref{e-gue180510}, we described the behavior of \(S_m(x,x_0)\) for \(x\) in an open neighborhood of \(x_0\) in \(X\) when \(m\) goes to infinity. As a consequence we obtained a (point-wise)  expansion for \(S_m(x_0,x_0)\). But with this method, we could not say anything about the asymptotic behavior of \(S_m(x,x)\) and its derivatives in \(x\) when \(m\) becomes large and \(x\) goes to \(x_0\).
In this work, we introduce some kind of gluing technique and we formulate and prove a result which covers that case (see Theorem~\ref{t-gue160605}) and in addition calculate the first three coefficients in the asymptotic expansion of \(S_m(x,x)\) (see Theorem~\ref{t-gue160605I}). Furthermore, the results in both works should be treated independently from each other since the methods of proof are  different. In \cite{HHL15} the starting point is the famous result on general Szeg\H{o} kernels for strongly pseudoconvex CR manifolds by Boutet de Monvel and Sj\"orstrand \cite{BouSj76}, while in this work we use a deep result on Bergman kernel expansion by Hsiao--Marinescu \cite{HM12} to obtain our results.
\end{rem}

We introduce now the geometric objects used in Theorem~\ref{t-gue160605I} below. The two  form $\frac{1}{2\pi}id\omega_0$ induces a rigid Hermitian metric $\langle\,\cdot\,|\,\cdot\,\rangle_{\mathcal{L}}$ on $\Complex TX$ (see \eqref{e-gue160529}). The Hermitian metric  $\langle\,\cdot\,|\,\cdot\,\rangle_{\mathcal{L}}$ on $\Complex TX$ induces a Hermitian metric on $\oplus_{r=1}^{2n+1}\Lambda^r(\Complex T^*X)$, also denoted by $\langle\,\cdot\,|\,\cdot\,\rangle_{\mathcal{L}}$.
For $u\in\Lambda^r(\Complex T^*X)$, we denote $\abs{u}^2_{\mathcal{L}}:=\langle\,u\,|\,u\,\rangle_{\mathcal{L}}$.  Let $x=(z,\theta)$ be canonical coordinates on an open set $D\subset X$ (see Theorem~\ref{t-gue150514}).
Let $Z_1\in C^\infty(D,T^{1,0}X),\ldots,Z_n\in C^\infty(D,T^{1,0}X)$ be as in \eqref{e-can} and let $e_1\in C^\infty(D,T^{*1,0}X),\ldots,e_n\in C^\infty(D,T^{*1,0}X)$ be the dual frames. The CR rigid Laplacian with respect to $\langle\,\cdot\,|\,\cdot\,\rangle_{\mathcal{L}}$ is given by
\begin{equation} \label{s1-e9m}
\triangle_{\mathcal{L}}=(-2)\sum^n_{j,k=1}\langle\,e_j\,|\,e_k\,\rangle_{\mathcal{L}}Z_j\ol{Z_k}.
\end{equation}
%(see \eqref{s1-e9} and the discusison after \eqref{s1-e9}).
%Let $x=(z,\theta)$ be canonical coordinates on an open set $D\subset X$ (see Theorem~\ref{t-gue150514}).
It is easy to check that $\triangle_{\mathcal{L}}$ is globally defined.
Let
\begin{equation}\label{e-gue160531m}
\frac{1}{n!}\Bigr((-\frac{1}{2\pi}d\omega_0)^n\wedge(-\omega_0)\Bigr)(x)=a(x)dx_1\cdots dx_{2n+1}\ \ \mbox{on $D$},
\end{equation}
where $a(x)\in C^\infty(D)$.
The rigid scalar curvature $S_{\mathcal{L}}$ is given by
\begin{equation}\label{e-gue160531Im}
S_{\mathcal{L}}(x):=\triangle_{\mathcal{L}}(\log a(x)).
\end{equation}
It is easy to see that $S_{\mathcal{L}}(x)$ is well-defined and  $S_{\mathcal{L}}(x)\in C^\infty(X)$  (see the discussion after \eqref{e-gue160531I}). We will show in Theorem~\ref{t-gue161012r} that $S_{\mathcal{L}}(x)=4\pi R$,
where $R$ denotes the Tanaka-Webster scalar curvature with respect to the pseudohermitian structure $-\omega_0$.

Let $\omega^1(x),\ldots,\omega^n(x)\in C^\infty(X,T^{*1,0}X)$ be an orthonormal basis for $T^{*1,0}_xX$ with respect to the given rigid Hermitian metrci $\langle\,\cdot\,|\,\cdot\,\rangle$, for every $x\in X$. Define $\Theta(x):=i\sum^n_{j=1}\omega^j(x)\wedge\ol{\omega^j}(x)\in C^\infty(X,T^{*1,1}X)$. Let $x=(z,\theta)$ be canonical coordinates on an open set $D\subset X$. Let
\begin{equation}\label{e-gue160601m}
\frac{1}{n!}\Theta^n \wedge(-\omega_0)=b(x)dx_1dx_2\cdots dx_{2n+1}\ \ \mbox{on $D$},
\end{equation}
where $b(x)\in C^\infty(D)$. Put
\begin{equation}\label{e-gue160601Im}
S^\Theta_{\mathcal{L}}(x):=\triangle_{\mathcal{L}}(\log b(x)),
\end{equation}
\begin{equation} \label{e-gue160601IIm}
\mathcal{R}^{\det}_\Theta(x)=\pr_b\ddbar_b\log b(x).
\end{equation}
We can check that $S^\Theta_{\mathcal{L}}(x)$ and $\mathcal{R}^{\rm det\,}_\Theta(x)$ are well-defined and $S^\Theta_{\mathcal{L}}(x)\in C^\infty(X)$, $TS^\Theta_{\mathcal{L}}(x)=0$,  $\mathcal{R}^{\det}_\Theta(x)\in C^\infty(X,T^{*1,1}X)$. We call $\mathcal{R}^{\det}_\Theta$ the rigid curvature of the determinant line bundle of $T^{*1,0}X$ with respect to the real two form $\Theta$. Note that $\langle\,\cdot\,|\,\cdot\,\rangle=\langle\,\cdot\,|\,\cdot\,\rangle_{\mathcal{L}}$ implies
$S^\Theta_{\mathcal{L}}(x)=S_{\mathcal{L}}(x)$.

Let $R^{T^{1,0}X}_{\mathcal{L}}\in C^\infty(X,T^{*1,1}X\otimes{\rm End\,}(T^{1,0}X))$ be the rigid Chern curvature with respect to $\langle\,\cdot\,|\,\cdot\,\rangle_{\mathcal{L}}$ (see \eqref{s1-e13}). Set
\begin{equation} \label{s1-e14m}
\abs{R^{T^{1,0}X}_{\mathcal{L}}}^2_{\mathcal{L}}:=\sum^n_{j,k,s,t=1}\abs{\langle\,R^{T^{1,0}X}_{\mathcal{L}}(\ol e_j,e_k)e_s\,|\,e_t\,\rangle_{\mathcal{L}}}^2,
\end{equation}
where $e_1,\ldots,e_n$ is an orthonormal frame for $T^{1,0}X$ with respect to $\langle\,\cdot\,|\,\cdot\,\rangle_{\mathcal{L}}$.
The rigid Ricci curvature with respect to $\langle\,\cdot\,|\,\cdot\,\rangle_{\mathcal{L}}$ is a global $(1,1)$ form on $X$ given by
\begin{equation}\label{e-gue160602am}
\langle\,{\rm Ric\,}_{\mathcal{L}}\,, \ol U\wedge V\,\rangle=-\sum^n_{j=1}\langle\,R^{T^{1,0}X}_{\mathcal{L}}(\ol U,e_j)V\,|\, e_j\,\rangle_{\mathcal{L}},\ \
U, V\in T^{1,0}X,
\end{equation}
where $e_1,\ldots,e_n$ is an orthonormal frame for $T^{1,0}X$ with respect to $\langle\,\cdot\,|\,\cdot\,\rangle_{\mathcal{L}}$.

We denote by $\dot{\mathcal{R}}=\dot{\mathcal{R}}(x)$ the Hermitian matrix $\dot{\mathcal{R}}(x)\in \operatorname{End}(T^{1,0}_xX)$ such that for $V, W \in T^{1,0}_xX$ we have
\begin{equation}\label{E:1.5.15m}
i d\omega_0(x)(V, \overline{W}) = \langle\,\dot{\mathcal{R}}(x)V\,|\, W \,\rangle.
\end{equation}
%Let $\det\dot{\mathcal{R}}(x)=\lambda_1(x)\cdots\lambda_{n}(x)$, where $\lambda_1(x),\ldots,\lambda_{n}(x)$ are eigenvalues of $\dot{\mathcal{R}}(x)$.
Now, we can state our result

\begin{thm}\label{t-gue160605I}
 With the notations used above, for $b_{0}(x), b_{1}(x), b_{2}(x)$ in Theorem~\ref{t-gue160605}, we have
\begin{equation} \label{e-gue160822}
b_{0}(x)=(2\pi)^{-n-1}\det\dot{\mathcal{R}}(x),\ \ \forall x\in X,
\end{equation}
\begin{equation} \label{e-gue160822I}
b_{1}(x)=(2\pi)^{-n-1}\det\dot{\mathcal{R}}(x)\Bigr(\frac{1}{4\pi}S^\Theta_{\mathcal{L}}-\frac{1}{8\pi}S_{\mathcal{L}}\Bigr)(x),\ \ \forall x\in X,
\end{equation}
\begin{equation} \label{e-gue160822II}
\begin{split}
b_{2}(x)&=(2\pi)^{-n-1}\det\dot{\mathcal{R}}(x)\Bigr(\frac{1}{128\pi^2}(S_{\mathcal{L}})^2-\frac{1}{32\pi^2}S_{\mathcal{L}}S^\Theta_{\mathcal{L}}+\frac{1}{32\pi^2}(S^\Theta_{\mathcal{L}})^2
-\frac{1}{32\pi^2}\triangle_{\mathcal{L}}S^\Theta_{\mathcal{L}}
-\frac{1}{8\pi^2}\abs{R^{\det}_\Theta}^2_{\mathcal{L}}\\
&\quad+\frac{1}{8\pi^2}\langle\,{\rm Ric\,}_{\mathcal{L}}\,|\,R^{\det}_\Theta\,\rangle_{\mathcal{L}}+\frac{1}{96\pi^2}\triangle_{\mathcal{L}}S_{\mathcal{L}}-
\frac{1}{24\pi^2}\abs{{\rm Ric\,}_{\mathcal{L}}}^2_{\mathcal{L}}+\frac{1}{96\pi^2}\abs{R^{T^{1,0}X}_{\mathcal{L}}}^2_{\mathcal{L}}\Bigr)(x),\  \forall x\in X.
\end{split}
\end{equation}
\end{thm}
In view of Remark \ref{Rem1},  Theorem 1.6 is a consequence of the calculation in \cite[Theorem 0.1]{MM12}.

This work is organized as follows. In Section~\ref{s:prelim} we introduce some basic notations and recall some definitions for CR manifolds with circle actions. Section~\ref{s-gue160527a} contains a more detailed description of the geometric quantities which appear in Theorem~\ref{t-gue160605I}. Furthermore, we show how some of them are related to geometric objects coming from pseudohermitian geometry (see Section~\ref{s-gue160603}, Theorem \ref{t-gue161012r}). In Section~\ref{s-gue160822} we prove  Theorem~\ref{t-gue160605} and Theorem~\ref{t-gue160605I}.

\section{Preliminaries}\label{s:prelim}

\subsection{Some standard notations}\label{s-gue150508b}
We use the following notations: $\mathbb N=\set{1,2,\ldots}$, $\mathbb N_0=\mathbb N\cup\set{0}$, $\Real$
is the set of real numbers, $\Real_+:=\set{x\in\Real;\, x>0}$, $\ol\Real_+:=\set{x\in\Real;\, x\geq0}$.
For a multiindex $\alpha=(\alpha_1,\ldots,\alpha_n)\in\mathbb N_0^n$
we set $\abs{\alpha}=\alpha_1+\cdots+\alpha_n$. For $x=(x_1,\ldots,x_n)$ we write
\[
\begin{split}
&x^\alpha=x_1^{\alpha_1}\ldots x^{\alpha_n}_n,\quad
 \pr_{x_j}=\frac{\pr}{\pr x_j}\,,\quad
\pr^\alpha_x=\pr^{\alpha_1}_{x_1}\ldots\pr^{\alpha_n}_{x_n}=\frac{\pr^{\abs{\alpha}}}{\pr x^\alpha}\,.
\end{split}
\]
Let $z=(z_1,\ldots,z_n)$, $z_j=x_{2j-1}+ix_{2j}$, $j=1,\ldots,n$, be coordinates of $\Complex^n$.
We write
\[
\begin{split}
&z^\alpha=z_1^{\alpha_1}\ldots z^{\alpha_n}_n\,,\quad\ol z^\alpha=\ol z_1^{\alpha_1}\ldots\ol z^{\alpha_n}_n\,,\\
&\pr_{z_j}=\frac{\pr}{\pr z_j}=
\frac{1}{2}\Big(\frac{\pr}{\pr x_{2j-1}}-i\frac{\pr}{\pr x_{2j}}\Big)\,,\quad\pr_{\ol z_j}=
\frac{\pr}{\pr\ol z_j}=\frac{1}{2}\Big(\frac{\pr}{\pr x_{2j-1}}+i\frac{\pr}{\pr x_{2j}}\Big),\\
&\pr^\alpha_z=\pr^{\alpha_1}_{z_1}\ldots\pr^{\alpha_n}_{z_n}=\frac{\pr^{\abs{\alpha}}}{\pr z^\alpha}\,,\quad
\pr^\alpha_{\ol z}=\pr^{\alpha_1}_{\ol z_1}\ldots\pr^{\alpha_n}_{\ol z_n}=
\frac{\pr^{\abs{\alpha}}}{\pr\ol z^\alpha}\,.
\end{split}
\]

Let $X$ be a $C^\infty$ orientable paracompact manifold.
We let $TX$ and $T^*X$ denote the tangent bundle of $X$ and the cotangent bundle of $X$ respectively.
The complexified tangent bundle of $X$ and the complexified cotangent bundle of $X$
will be denoted by $\Complex TX$ and $\Complex T^*X$ respectively. We write $\langle\,\cdot\,,\cdot\,\rangle$
to denote the pointwise duality between $T^*X$ and $TX$.
We extend $\langle\,\cdot\,,\cdot\,\rangle$ bilinearly to $\Complex T^*X\times\Complex TX$. For $u\in \Complex T^*X$, $v\in\Complex TX$, we also write $u(v):=\langle\,u\,,v\,\rangle$.

Let $E$ be a $C^\infty$ vector bundle over $X$. The fiber of $E$ at $x\in X$ will be denoted by $E_x$.
Let $F$ be another vector bundle over $X$. We write
$F\boxtimes E^*$ to denote the vector bundle over $X\times X$ with fiber over $(x, y)\in X\times X$
consisting of the linear maps from $E_y$ to $F_x$.

Let $Y\subset X$ be an open set. The spaces of
smooth sections of $E$ over $Y$ and distribution sections of $E$ over $Y$ will be denoted by $C^\infty(Y, E)$ and $\mathscr D'(Y, E)$ respectively.
Let $\mathscr E'(Y, E)$ be the subspace of $\mathscr D'(Y, E)$ whose elements have compact support in $Y$.
For $m\in\Real$, we let $H^m(Y, E)$ denote the Sobolev space
of order $m$ of sections of $E$ over $Y$. Put
\begin{gather*}
H^m_{\rm loc\,}(Y, E)=\big\{u\in\mathscr D'(Y, E);\, \varphi u\in H^m(Y, E),
      \,\forall \varphi\in C^\infty_0(Y)\big\}\,,\\
      H^m_{\rm comp\,}(Y, E)=H^m_{\rm loc}(Y, E)\cap\mathscr E'(Y, E)\,.
\end{gather*}

\subsection{Definitions and notations from semi-classical analysis} \label{s-gue160902}

Let $\Omega$ be a $C^\infty$ paracompact manifold equipped with a smooth density of integration.
Any continuous linear operator
$A:C^\infty_0(\Omega)\To \mathscr D'(\Omega)$ has a Schwartz distribution kernel, denoted
$A(x, y)\in\mathscr D'(\Omega\times\Omega)$.
We say that $A$ is a \emph{smoothing operator} if $A(x,y)\in C^\infty(\Omega\times\Omega)$.
We say that $A$ is properly supported if ${\rm Supp\,}A(x,y)\subset\Omega\times\Omega$ is proper. That is, the two projections: $t_x:(x,y)\in{\rm Supp\,}A(x,y)\To x\in\Omega$, $t_y:(x,y)\in{\rm Supp\,}A(x,y)\To y\in\Omega$ are proper (i.e. the inverse images of $t_x$ and $t_y$ of all compact subsets of $\Omega$ are compact).

Let $H(x,y)\in\mathscr D'(\Omega\times\Omega)$. We write $H$ to denote the unique continuous operator $H:C^\infty_0(\Omega)\To\mathscr D'(\Omega)$ with distribution kernel $H(x,y)$. In this work, we identify $H$ with $H(x,y)$.

Let $W_1$, $W_2$ be open sets in $\Real^N$.
An $m$-dependent continuous operator
$A_m:C^\infty_0(W_1)\To\mathscr D'(W_2)$ is called $m$-negligible on $W_2\times W_1$
if for $m$ large enough $A_m$ is smoothing  and for any $K\Subset W_2\times W_1$, any
multi-indices $\alpha$, $\beta$ and any $N\in\mathbb N$ there exists $C_{K,\alpha,\beta,N}>0$
such that
\begin{equation}\label{e-gue13628III}
\abs{\pr^\alpha_x\pr^\beta_yA_m(x, y)}\leq C_{K,\alpha,\beta,N}m^{-N}\:\: \text{on $K$}.
\end{equation}
In that case we write
\[A_m(x,y)=O(m^{-\infty})\:\:\text{on $W_2\times W_1$,}\]
or
\[A_m=O(m^{-\infty})\:\:\text{on $W_2\times W_1$.}\]
If $A_m, B_m:C^\infty_0(W_1)\To\mathscr D'(W_2)$ are $m$-dependent continuous operators,
we write $A_m= B_m+O(m^{-\infty})$ on $W_2\times W_1$ or $A_m(x,y)=B_m(x,y)+O(m^{-\infty})$ on $W_2\times W_1$ if $A_m-B_m=O(k^{-\infty})$ on $W_2\times W_1$.

Let $A_m, B_m:C^\infty(X)\To C^\infty(X)$ be $m$-dependent smoothing operators. We write $A_m=B_m+O(m^{-\infty})$ on $X\times X$ if on every local coordinate patch $D$, $A_m=B_m+O(m^{-\infty})$ on $D\times D$.

Let $F_m:H^s(X)\To H^{s'}(X)$ be a $m$-dependent continuous operator, where $s, s'\in\Real$. We write
\[F_m=O(m^{n_0}):H^s(X)\To H^{s'}(X),\ \ n_0\in\mathbb Z,\]
if there is a positive constant $c$ independent of $m$, such that
\begin{equation} \label{s1-e1su}
\norm{F_mu}_{s'}\leq cm^{n_0}\norm{u}_{s},\ \ \forall u\in H^s(X),
\end{equation}
where $\norm{\cdot}_s$ denotes the usual Sobolev norm on $X$ of order $s$. It is easy to check that
$F_m=O(m^{-\infty})$ on $X\times X$ if and only if $F_m=O(m^{-N}):H^{-s}(X)\To H^{s}(X)$, for every $N>0$ and $s\in\mathbb N$.

We recall the definition of the semi-classical symbol spaces

\begin{defn} \label{d-gue140826}
Let $W$ be an open set in $\Real^N$. Let
%$S(1;W)=S(1)$ be the set of
%$a\in \cC^\infty(W)$ such that for every $\alpha\in\mathbb N^N_0$, there
%exists $C_\alpha>0$, such that $\abs{\pr^\alpha_xa(x)}\leq
%C_\alpha$ on $W$.
\begin{gather*}
S(1;W):=\Big\{a\in C^\infty(W)\,|\, \forall\alpha\in\mathbb N^N_0:
\sup_{x\in W}\abs{\pr^\alpha a(x)}<\infty\Big\},\\
S^0_{{\rm loc\,}}(1;W):=\Big\{(a(\cdot,m))_{m\in\Real}\,|\, \forall\alpha\in\mathbb N^N_0,
\forall \chi\in C^\infty_0(W)\,:\:\sup_{m\in\Real, m\geq1}\sup_{x\in W}\abs{\pr^\alpha(\chi a(x,m))}<\infty\Big\}\,.
\end{gather*}
%The space $S_{{\rm loc\,}}(1;W)$ is the set of sequences $a=(a(\cdot,k))$ in $\cC^\infty(W)$
%with the property that for any $\chi\in\cC^\infty_0(W)$ we have
%%\[
%$\sup_{k\in\N}\sup_{x\in W}\abs{\pr^\alpha a(x,k)}<\infty$\,.
%%\]
%If $a=a(x,k)$ depends on $k\in]1,\infty[$, we say that
%$a(x,k)\in S_{{\rm loc\,}}(1;W)=S_{{\rm loc\,}}(1)$ if $\chi(x)a(x,k)$ is uniformly bounded
%in $S(1)$ when $k$ varies in $]1,\infty[$, for any $\chi\in\cC^\infty_0(W)$.
For $k\in\Real$ let
\[
S^k_{{\rm loc}}(1):=S^k_{{\rm loc}}(1;W)=\Big\{(a(\cdot,m))_{m\in\Real}\,|\,(m^{-k}a(\cdot,m))\in S^0_{{\rm loc\,}}(1;W)\Big\}\,.
\]
%For $m\in\Real$, we put $S^m_{{\rm loc}}(1;W)=\{(a(\cdot,k)):(k^{-m}a(\cdot,k))\in S_{{\rm loc\,}}(1;W)\}$.
%For $m\in\Real$, we put $S^m_{{\rm loc}}(1;W)=S^m_{{\rm loc}}(1)=k^mS_{{\rm loc\,}}(1)$.
Hence $a(\cdot,m)\in S^k_{{\rm loc}}(1;W)$ if for every $\alpha\in\mathbb N^N_0$ and $\chi\in C^\infty_0(W)$, there
exists $C_\alpha>0$ independent of $m$, such that $\abs{\pr^\alpha (\chi a(\cdot,m))}\leq C_\alpha m^{k}$ holds on $W$.

Consider a sequence $a_j\in S^{k_j}_{{\rm loc\,}}(1)$, $j\in\N_0$, where $k_j\searrow-\infty$,
and let $a\in S^{k_0}_{{\rm loc\,}}(1)$. We say
\[
a(\cdot,m)\sim
\sum\limits^\infty_{j=0}a_j(\cdot,m)\:\:\text{in $S^{k_0}_{{\rm loc\,}}(1)$},
\]
if for every
$\ell\in\N_0$ we have $a-\sum^{\ell}_{j=0}a_j\in S^{k_{\ell+1}}_{{\rm loc\,}}(1)$ .
For a given sequence $a_j$ as above, we can always find such an asymptotic sum
$a$, which is unique up to an element in
$S^{-\infty}_{{\rm loc\,}}(1)=S^{-\infty}_{{\rm loc\,}}(1;W):=\cap _kS^k_{{\rm loc\,}}(1)$.
\end{defn}
%We say that $a(\cdot,m)\in S^{k}_{{\rm loc\,}}(1)$ is a classical symbol on $W$ of order $k$ if
%\begin{equation} \label{e-gue13628I}
%a(\cdot,m)\sim\sum\limits^\infty_{j=0}m^{k-j}a_j\: \text{in $S^{k}_{{\rm loc\,}}(1)$},\ \ a_j(x)\in
%S^0_{{\rm loc\,}}(1),\ j=0,1\ldots.
%\end{equation}
%The set of all classical symbols on $W$ of order $k$ is denoted by
%$S^{k}_{{\rm loc\,},{\rm cl\,}}(1)=S^{k}_{{\rm loc\,},{\rm cl\,}}(1;W)$.
%\end{defn}
%--------------

\subsection{Set up and terminology}\label{s-gue150508bI}

Let $(X, T^{1,0}X)$ be a compact CR manifold of dimension $2n+1$, $n\geq 1$, where $T^{1,0}X$ is a CR structure of $X$. That is $T^{1,0}X$ is a subbundle of rank $n$ of the complexified tangent bundle $\mathbb{C}TX$, satisfying $T^{1,0}X\cap T^{0,1}X=\{0\}$, where $T^{0,1}X=\overline{T^{1,0}X}$, and $[\mathcal V,\mathcal V]\subset\mathcal V$, where $\mathcal V=C^\infty(X, T^{1,0}X)$. We assume that $X$ admits an $S^1$ action: $S^1\times X\rightarrow X$. We write $e^{i\theta}$ to denote the $S^1$ action. Let $T\in C^\infty(X, TX)$ be the global real vector field induced by the $S^1$ action given by
$(Tu)(x)=\frac{\partial}{\partial\theta}\left(u(e^{i\theta}\circ x)\right)|_{\theta=0}$, $u\in C^\infty(X)$.

\begin{defn}\label{d-gue160502}
We say that the $S^1$ action $e^{i\theta}$ is CR if
$[T, C^\infty(X, T^{1,0}X)]\subset C^\infty(X, T^{1,0}X)$ and the $S^1$ action is transversal if for each $x\in X$,
$\Complex T(x)\oplus T_x^{1,0}X\oplus T_x^{0,1}X=\mathbb CT_xX$. Moreover, we say that the $S^1$ action is locally free if $T\neq0$ everywhere. It should be mentioned that transversality implies locally free.
\end{defn}

We assume throughout that $(X, T^{1,0}X)$ is a compact connected CR manifold with a transversal CR locally free $S^1$ action $e^{i\theta}$ and we let $T$ be the global vector field induced by the $S^1$ action. Let $\omega_0\in C^\infty(X,T^*X)$ be the global real one form determined by $\langle\,\omega_0\,,\,u\,\rangle=0$, for every $u\in T^{1,0}X\oplus T^{0,1}X$ and $\langle\,\omega_0\,,\,T\,\rangle=-1$.
%For $x\in X$, we say that the period of $x$ is $\frac{2\pi}{\ell}$, $\ell\in\mathbb N$, if $e^{i\theta}\circ x\neq x$, for every $0<\theta<\frac{2\pi}{\ell}$ and $e^{i\frac{2\pi}{\ell}}\circ x=x$. For each $\ell\in\mathbb N$, put
%\begin{equation}\label{e-gue150802b}
%X_\ell=\set{x\in X;\, \mbox{the period of $x$ is $\frac{2\pi}{\ell}$}}
%\end{equation}
%and let $p=\min\set{\ell\in\mathbb N;\, X_\ell\neq\emptyset}$. It is well-known that if $X$ is connected, then $X_p$ is an open and dense subset of $X$ (see Duistermaat-Heckman~\cite{Du82}).
%Assume $X=X_{p_1}\bigcup X_{p_2}\bigcup\cdots\bigcup X_{p_k}$ (see \eqref{e-gue150802bm}), $p=:p_1<p_2<\cdots<p_k$. In this work, we assume that $p_1=1$
%and we denote $X_{{\rm reg\,}}:=X_{p_1}=X_1$.
%We call $x\in X_{{\rm reg\,}}$ a regular point of the $S^1$ action. Let $X_{{\rm sing\,}}$ be the complement of $X_{{\rm reg\,}}$.

\begin{defn}\label{d-gue150508f}
For $p\in X$, the Levi form $\mathcal L_p$ is the Hermitian quadratic form on $T^{1,0}_pX$ given by
$\mathcal{L}_p(U,\ol V)=-\frac{1}{2i}\langle\,d\omega_0(p)\,,\,U\wedge\ol V\,\rangle$, $U, V\in T^{1,0}_pX$.
\end{defn}

If the Levi form $\mathcal{L}_p$ is positive definite, we say that $X$ is strongly pseudoconvex at $p$. If the Levi form is positive definite at every point of $X$, we say that $X$ is strongly pseudoconvex. We assume throughout that $(X, T^{1,0}X)$ is strongly pseudoconvex.

Denote by $T^{*1,0}X$ and $T^{*0,1}X$ the dual bundles of
$T^{1,0}X$ and $T^{0,1}X$ respectively. Define the vector bundle of $(p,q)$ forms by
$T^{*p,q}X=\Lambda^p(T^{*1,0}X)\wedge\Lambda^q(T^{*0,1}X)$. Let $\Omega^{p,q}(D)$
denote the space of smooth sections of $T^{*p,q}X$ over $D$ and let $\Omega_0^{p,q}(D)$
be the subspace of $\Omega^{p,q}(D)$ whose elements have compact support in $D$.
For $p=q=0$, we write $C^\infty(D):=\Omega^{0,0}(D)$ and $C^\infty_0(D):=\Omega^{0,0}_0(D)$.

Fix $\theta_0\in]-\pi, \pi[$, $\theta_0$ small. Let
$$d e^{i\theta_0}: \mathbb CT_x X\rightarrow \mathbb CT_{e^{i\theta_0}x}X$$
denote the differential map of $e^{i\theta_0}: X\rightarrow X$. By the CR property of the $S^1$ action, we can check that
\begin{equation}\label{e-gue150508fa}
\begin{split}
de^{i\theta_0}:T_x^{1,0}X\rightarrow T^{1,0}_{e^{i\theta_0}x}X,\\
de^{i\theta_0}:T_x^{0,1}X\rightarrow T^{0,1}_{e^{i\theta_0}x}X,\\
de^{i\theta_0}(T(x))=T(e^{i\theta_0}x).
\end{split}
\end{equation}
Let $(e^{i\theta_0})^*:\Lambda^r(\Complex T^*X)\To\Lambda^r(\Complex T^*X)$ be the pull-back map by $e^{i\theta_0}$, $r=0,1,\ldots,2n+1$. From \eqref{e-gue150508fa}, it is easy to see that for every $q=0,1,\ldots,n$ one has
\begin{equation}\label{e-gue150508faI}
(e^{i\theta_0})^*:T^{*0,q}_{e^{i\theta_0}x}X\To T^{*0,q}_{x}X.
\end{equation}
Let $u\in\Omega^{0,q}(X)$ be arbitrary. Define
\begin{equation}\label{e-gue150508faII}
Tu:=\frac{\pr}{\pr\theta}\bigr((e^{i\theta})^*u\bigr)|_{\theta=0}\in\Omega^{0,q}(X).
\end{equation}
(See also \eqref{lI}.) For every $\theta\in\Real$ and every $u\in C^\infty(X,\Lambda^r(\Complex T^*X))$, we write $u(e^{i\theta}\circ x):=(e^{i\theta})^*u(x)$. It is clear that for every $u\in C^\infty(X,\Lambda^r(\Complex T^*X))$, we have
\begin{equation}\label{e-gue150510f}
u(x)=\sum_{m\in\mathbb Z}\frac{1}{2\pi}\int^{\pi}_{-\pi}u(e^{i\theta}\circ x)e^{-im\theta}d\theta.
\end{equation}

For every $m\in\mathbb Z$, let
\begin{equation}\label{e-gue150508dI}
\Omega^{0,q}_m(X):=\set{u\in\Omega^{0,q}(X);\, Tu=imu},\ \ q=0,1,2,\ldots,n.
\end{equation}
We denote $C^\infty_m(X):=\Omega^{0,0}_m(X)$.

Let $\ddbar_b:\Omega^{0,q}(X)\rightarrow\Omega^{0,q+1}(X)$ be the tangential Cauchy-Riemann operator. From the CR property of the $S^1$ action, it is straightforward to see that (see also \eqref{e-gue150514f})
\[T\ddbar_b=\ddbar_bT\ \ \mbox{on $\Omega^{0,q}(X)$}\]
holds. Hence,
\begin{equation}\label{e-gue160527}
\ddbar_b:\Omega^{0,q}_m(X)\To\Omega^{0,q+1}_m(X),\ \ \forall m\in\mathbb Z.
\end{equation}

\begin{defn}\label{d-gue50508d}
Let $D\subset U$ be an open set. We say that a function $u\in C^\infty(D)$ is rigid if $Tu=0$. We say that a function $u\in C^\infty(X)$ is Cauchy-Riemann (CR for short)
if $\ddbar_bu=0$. We call $u$ a rigid CR function if  $\ddbar_bu=0$ and $Tu=0$.
\end{defn}
\begin{defn} \label{d-gue150508dI}
Let $F$ be a complex vector bundle over $X$. We say that $F$ is rigid (resp.\ CR, resp.\ rigid CR) if there exists
an open cover $(U_j)_j$ of $X$ and trivializing frames $\set{f^1_j,f^2_j,\dots,f^r_j}$ on $U_j$,
such that the corresponding transition matrices are rigid (resp.\ CR, resp.\ rigid CR).
\end{defn}

Let $F$ be a rigid (CR) vector bundle over $X$. In the following, we fix an open cover $(U_j)^N_{j=1}$ of $X$ and  a family $\set{f^1_j, f^2_j, \dots, f^r_j}^N_{j=1}$
of trivializing frames $\set{f^1_j,f^2_j,\dots,f^r_j}$
on each $U_j$ such that the entries of the transition matrices between different frames $\set{f^1_j, f^2_j,\dots, f^r_j}$ are rigid (CR). For any local trivializing frame $\set{f^1,\ldots,f^r}$ of $F$ on an open set $D$, we say that $\set{f^1,\ldots,f^r}$ is a rigid (CR) frame if the entries of the transition matrices between $\set{f^1,\ldots,f^r}$ and $\set{f^1_j,\ldots,f^r_j}$ are rigid (CR), for every $j$, and we call $D$ a local rigid (CR) trivialization. We can define the operator $T$ on $\Omega^{0,q}(X,F)$ in the standard way.

\begin{defn}\label{d-gue150514f}
Let $F$ be a complex rigid vector bundle over $X$ and let $\langle\,\cdot\,|\,\cdot\,\rangle_F$ be a Hermitian metric on $F$. We say that $\langle\,\cdot\,|\,\cdot\,\rangle_F$ is a rigid Hermitian metric if for every rigid local frames $f_1,\ldots, f_r$ of $F$, we have $T\langle\,f_j\,|\,f_k\,\rangle_F=0$, for every $j,k=1,2,\ldots,r$.
\end{defn}

We notice that Definition~\ref{d-gue150514f} above depends on the fixed open cover $(U_j)^N_{j=1}$ of $X$ and  a family $\set{f^1_j,f^2_j,\dots,f^r_j}^N_{j=1}$ of trivializing frames $\set{f^1_j,f^2_j,\dots,f^r_j}$ on each $U_j$ such that the entries of the transition matrices between different frames $\set{f^1_j,f^2_j,\dots,f^r_j}$ are rigid.

It is well-known that there is a rigid Hermitian metric on any rigid vector bundle $F$ (see Theorem 2.10 in~\cite{CHT} and Theorem 10.5 in~\cite{Hsiao14}). Note that  Baouendi-Rothschild-Treves~\cite{BRT85} proved that $T^{1,0}X$ is a rigid complex vector bundle over $X$. More precisely, by arranging an atlas of BRT trivializations, it is easy to see that  that $T^{1, 0}X$ is a rigid CR vector bundle over $X$. In this work, we will fix an open BRT trivialization cover $(U_j)^N_{j=1}$ of $X$ and  a family
of BRT frames $\set{Z_j}^n_{j=1}$ on each $U_j$
(see Theorem~\ref{t-gue150514} for the definitions of BRT trivializations and BRT frames).

From now on, take a rigid Hermitian metric $\langle\,\cdot\,|\,\cdot\,\rangle$ on $\Complex TX$ such that $T^{1,0}X\perp T^{0,1}X$, $T\perp (T^{1,0}X\oplus T^{0,1}X)$, $\langle\,T\,|\,T\,\rangle=1$.
We denote by $dv_X=dv_X(x)$ the volume form on $X$ induced by the fixed
Hermitian metric $\langle\,\cdot\,|\,\cdot\,\rangle$ on $\Complex TX$. Then we get natural global $L^2$ inner products $(\,\cdot\,|\,\cdot\,)$ on $\Omega^{0,q}(X)$. We denote by $L^2_{(0,q)}(X)$ the completion of $\Omega^{0,q}(X)$  with respect to $(\,\cdot\,|\,\cdot\,)$. For $f\in L^{2}_{(0,q)}(X)$, we denote $\norm{f}^2:=(\,f\,|\,f\,)$. For each $m\in\mathbb Z$, we denote by $L^2_{(0,q),m}(X)$ the completion of $\Omega^{0,q}_m(X)$ with respect to $(\,\cdot\,|\,\cdot\,)$. For $q=0$, we write $L^2(X):=L^2_{(0,0)}(X)$, $L^2_m(X):=L^2_{(0,0),m}(X)$.

\section{Rigid CR Geometry}\label{s-gue160527a}

In this section, we will introduce some rigid geometric functions on CR manifolds with $S^1$ action which appear in Theorem \ref{t-gue160605I}.

We would like to introduce these geometric quantities via local computations. Therefore we need the following result on local coordinates due to Baouendi-Rothschild-Treves~\cite{BRT85}.

\begin{thm}\label{t-gue150514}
For every point $x_0\in X$, we can find local coordinates $x=(x_1,\cdots,x_{2n+1})=(z,\theta)=(z_1,\cdots,z_{n},\theta), z_j=x_{2j-1}+ix_{2j},j=1,\cdots,n, x_{2n+1}=\theta$, defined in some small neighborhood $D=\{(z, \theta): \abs{z}<\delta, -\varepsilon_0<\theta<\varepsilon_0\}$ of $x_0$, $\delta>0$, $0<\varepsilon_0<\pi$, such that $(z(x_0),\theta(x_0))=(0,0)$ and
\begin{equation}\label{e-can}
\begin{split}
&T=\frac{\partial}{\partial\theta}\\
&Z_j=\frac{\partial}{\partial z_j}+i\frac{\partial\varphi}{\partial z_j}(z)\frac{\partial}{\partial\theta},j=1,\cdots,n
\end{split}
\end{equation}
where $Z_j(x), j=1,\cdots, n$, form a basis of $T_x^{1,0}X$, for each $x\in D$ and $\varphi(z)\in C^\infty(D,\mathbb R)$ independent of $\theta$. We call $(D,(z,\theta),\varphi)$ BRT trivialization, $x=(z,\theta)$ canonical coordinates and  $\set{Z_j}^n_{j=1}$ BRT frames.
\end{thm}
In order to prove that the geometric objects which will arise from local calculations using BRT trivializations are independent of the choice of such coordinates we have to understand the transition maps between BRT trivialisations. For the following result see  Theorem~II.1 and Proposition~I.2 in \cite{BRT85}.

\begin{lem}\label{l-gue161012}
Let $(D,(z,\theta),\varphi)$ be a BRT trivialization and let
$y=(y_1,\ldots,y_{2n+1})=(w,\gamma)$, $w_j=y_{2j-1}+iy_{2j}$, $j=1,\ldots,n$, $\gamma=y_{2n+1}$, be another canonical coordinates on $D$. Write
\begin{equation}\label{sp3-eVI}\begin{split}
&T=\frac{\pr}{\pr\gamma},\\
&\Td Z_j=\frac{\pr}{\pr w_j}+i\frac{\pr\Td\varphi}{\pr w_j}(w)\frac{\pr}{\pr\gamma},\ \ j=1,\ldots,n,
\end{split}
\end{equation}
where $\Td Z_j(y)$, $j=1,\ldots,n$, form a basis of $T^{1,0}_yX$, for each $y\in D$, and $\Td\varphi(w)\in C^\infty(D,\Real)$ independent of $\gamma$. Then,
%From \eqref{sp3-eVI} and \eqref{e-can}, it is not difficult to see that on $D$, we have
\begin{equation}\label{sp3-eVII}
\begin{split}
&w=(w_1,\ldots,w_{n})=(H_1(z),\ldots,H_{n}(z))=H(z),\ \ H_j(z)\in C^\infty,\ \ \forall j,\\
&\gamma=\theta+G(z),\ \ G(z)\in C^\infty,
\end{split}
\end{equation}
where for each $j=1,\ldots,n$, $H_j(z)$ is holomorphic.
\end{lem}

From \eqref{sp3-eVII}, we can check that
\begin{equation}\label{spca}
d\ol w_j=\sum^{n}_{l=1}\ol{\left(\frac{\pr H_j}{\pr z_l}\right)}d\ol z_l,\ \ j=1,\ldots,n.
\end{equation}

\begin{rem}\label{r-gue160529}
By using BRT trivialization, we get another way to define $Tu, \forall u\in\Omega^{0,q}(X)$. Let $(D,(z,\theta),\varphi)$ be a BRT trivialization. It is clear that
$$\{d\overline{z_{j_1}}\wedge\cdots\wedge d\overline{z_{j_q}}, 1\leq j_1<\cdots<j_q\leq n\}$$
is a basis for $T^{\ast0,q}_xX$, for every $x\in D$. Let $u\in\Omega^{0,q}(X)$. On $D$, we write
\begin{equation}\label{e-gue150524fb}
u=\sum\limits_{1\leq j_1<\cdots<j_q\leq n}u_{j_1\cdots j_q}d\overline{z_{j_1}}\wedge\cdots\wedge d\overline{z_{j_q}}.
\end{equation}
Then on $D$ we can check that
\begin{equation}\label{lI}
Tu=\sum\limits_{1\leq j_1<\cdots<j_q\leq n}(Tu_{j_1\cdots j_q})d\overline{z_{j_1}}\wedge\cdots\wedge d\overline{z_{j_q}}
\end{equation}
and $Tu$ is independent of the choice of BRT trivializations. Note that on BRT trivialization $(D,(z,\theta),\varphi)$, we have
\begin{equation}\label{e-gue150514f}
\ddbar_b=\sum^n_{j=1}d\ol z_j\wedge(\frac{\partial}{\partial\ol z_j}-i\frac{\partial\varphi}{\partial\ol z_j}(z)\frac{\partial}{\partial\theta}).
\end{equation}
\end{rem}

Let us return to our situation. We now introduce the geometric objects used in our main result Theorem~\ref{t-gue160605I}. The Levi form $\mathcal{L}_p$ induces a Hermitian metric $\langle\,\cdot\,|\,\cdot\,\rangle_{\mathcal{L}}$ on $\Complex TX$ given by
\begin{equation}\label{e-gue160529}
\begin{split}
&\langle\,U\,|\,V\,\rangle_{\mathcal{L}}=\frac{1}{2\pi}\langle\,id\omega_0\,,\,U\wedge\ol V\,\rangle,\  \ \forall U, V\in T^{1,0}X,\\
&T^{1,0}X\perp T^{0,1}X\perp\Complex T,\\
&\langle\,T\,|\,T\,\rangle_{\mathcal{L}}=1,\\
&\langle\,\ol U\,|\,\ol V\rangle_{\mathcal{L}}=\ol{\langle\,U\,|\,V\,\rangle_{\mathcal{L}}},\ \ \forall U, V\in T^{0,1}X.
\end{split}
\end{equation}
The Hermitian metric  $\langle\,\cdot\,|\,\cdot\,\rangle_{\mathcal{L}}$ on $\Complex TX$ induces a Hermitian metric  on $\oplus_{r=1}^{2n+1}\Lambda^r(\Complex T^*X)$, also denoted by $\langle\,\cdot\,|\,\cdot\,\rangle_{\mathcal{L}}$.
For $u\in\Lambda^r(\Complex T^*X)$, we denote $\abs{u}^2_{\mathcal{L}}:=\langle\,u\,|\,u\,\rangle_{\mathcal{L}}$.

%Let $Z_1,\ldots,Z_n$ be local rigid frames for $T^{1,0}X$ with dual frames $e_1,\ldots,e_n$. The CR rigid Laplacian with respect to $\langle\,\cdot\,|\,\cdot\,\rangle_{\mathcal{L}}$ is given by
Let $\triangle_{\mathcal{L}}$ be as in \eqref{s1-e9m}. Let $(D,(z,\theta),\varphi)$ be a BRT trivialization. We have
%\begin{equation} \label{s1-e9}
%\triangle_{\mathcal{L}}=(-2)\sum^n_{j,k=1}\langle\,e_j\,|\,e_k\,\rangle_{\mathcal{L}}Z_j\ol{Z_k}.
%\end{equation}
%It is easy to check that $\triangle_{\mathcal{L}}$ is well-defined.
%On a BRT trivialization $(D,(z,\theta),\varphi)$, we can check that
\begin{equation} \label{e-gue160531II}
\triangle_{\mathcal{L}}=(-2)\sum^n_{j,k=1}\langle\,dz_j\,|\,dz_k\,\rangle_{\mathcal{L}}\Bigr(\frac{\partial}{\partial z_j}+i\frac{\partial\varphi}{\partial z_j}(z)\frac{\partial}{\partial\theta}\Bigr)\Bigr(\frac{\partial}{\partial\ol z_k}-i\frac{\partial\varphi}{\partial\ol z_k}(z)\frac{\partial}{\partial\theta}\Bigr).
\end{equation}
Let $x=(z,\theta)$ be canonical coordinates on an open set $D\subset X$. Let
\begin{equation}\label{e-gue160531}
\frac{1}{n!}\Bigr((-\frac{1}{2\pi}d\omega_0)^n\wedge(-\omega_0)\Bigr)(x)=a(x)dx_1\cdots dx_{2n+1}\ \ \mbox{on $D$},
\end{equation}
where $a(x)\in C^\infty(D)$.
The rigid scalar curvature $S_{\mathcal{L}}$ is given by
\begin{equation}\label{e-gue160531I}
S_{\mathcal{L}}(x):=\triangle_{\mathcal{L}}(\log a(x)).
\end{equation}
Let
$y=(y_1,\ldots,y_{2n+1})=(w,\gamma)$, $w_j=y_{2j-1}+iy_{2j}$, $j=1,\ldots,n$, $\gamma=y_{2n+1}$, be another canonical coordinates on $D$. Let
\begin{equation}\label{e-gue160531a}
\frac{1}{n!}\Bigr((-\frac{1}{2\pi}d\omega_0)^n\wedge(-\omega_0)\Bigr)(y)=\Td a(y)dy_1\cdots dy_{2n+1}\ \ \mbox{on $D$},
\end{equation}
where $\Td a(y)\in C^\infty(D)$.
From \eqref{sp3-eVII}, we can check that
\begin{equation}\label{e-gue160531aI}
\Td a(y)\abs{\det\left(\frac{\pr H_j}{\pr z_k}\right)^n_{j,k=1}}^2=a(x),\ \, y=(w,\gamma),\ \ x=(z,\theta),\ \ w=H(z),\ \ \gamma=\theta+G(z).\end{equation}
From \eqref{e-gue160531II}, \eqref{e-gue160531aI} and %notice that
\[\triangle_{\mathcal{L}}(\log \abs{\det\left(\frac{\pr H_j}{\pr z_k}\right)^n_{j,k=1}}^2)=0,\]
we can check that
$\triangle_{\mathcal{L}}(\log a(x))=\triangle_{\mathcal{L}}(\log\Td a(y))$ holds on $D$ and hence the rigid scalar curvature $S_{\mathcal{L}}(x)$ is well-defined with $S_{\mathcal{L}}(x)\in C^\infty(X)$, $TS_{\mathcal{L}}(x)=0$.

%Let $\omega^1(x),\ldots,\omega^n(x)\in C^\infty(X,T^{*1,0}X)$ be an orthonormal basis for $T^{*1,0}_xX$ with respect to the given rigid Hermitian metrci $\langle\,\cdot\,|\,\cdot\,\rangle$, for every $x\in X$. Define $\Theta(x):=i\sum^n_{j=1}\omega^j(x)\wedge\ol{\omega^j}(x)\in C^\infty(X,T^{*1,1}X)$. Let $x=(z,\theta)$ be canonical coordinates on an open set $D\subset X$. Let
%\begin{equation}\label{e-gue160601}
%\frac{1}{n!}\Theta^n \wedge(-\omega_0)=b(x)dx_1dx_2\cdots dx_{2n+1}\ \ \mbox{on $D$},
%\end{equation}
%where $b(x)\in C^\infty(D)$. Put
%\begin{equation}\label{e-gue160601I}
%S^\Theta_{\mathcal{L}}(x):=\triangle_{\mathcal{L}}(\log b(x)),
%\end{equation}
%\begin{equation} \label{e-gue160601II}
%\mathcal{R}^{\det}_\Theta(x)=-\ddbar_b\pr_b\log b(x).
%\end{equation}
Let $S^\Theta_{\mathcal{L}}(x)$ and $\mathcal{R}^{\rm det\,}_\Theta(x)$ be as in \eqref{e-gue160601Im} and \eqref{e-gue160601IIm} respectively. We can repeat the procedure above with minor change and check that $S^\Theta_{\mathcal{L}}(x)$ and $\mathcal{R}^{\rm det\,}_\Theta(x)$ are well-defined.
% and $S^\Theta_{\mathcal{L}}(x)\in C^\infty(X)$, $TS^\Theta_{\mathcal{L}}(x)=0$,  $\mathcal{R}^{\det}_\Theta(x)\in C^\infty(X,T^{*1,1}X)$. We call $\mathcal{R}^{\det}_\Theta$ the rigid curvature of the determinant line bundle of $T^{*1,0}X$ with respect to the real two form $\Theta$. Note that when $\langle\,\cdot\,|\,\cdot\,\rangle=\langle\,\cdot\,|\,\cdot\,\rangle_{\mathcal{L}}$, $S^\Theta_{\mathcal{L}}(x)=S_{\mathcal{L}}(x)$.

The rigid Chern connection
\begin{equation}\label{e-gue160602}
\nabla^{T^{1,0}X}_{\mathcal{L}}:C^\infty(X,T^{1,0}X)\To C^\infty(X,\Complex T^*X\otimes T^{1,0}X)
\end{equation}
on $T^{1,0}X$ with respect to $\langle\,\cdot\,|\,\cdot\,\rangle_{\mathcal{L}}$ is defined as follows.  Let $(D,(z,\theta),\varphi)$ be a BRT trivialization and let $Z_j$, $j=1,\ldots,n$, be as in \eqref{e-can}. Set
\begin{equation}
h_{j,k}=h_{j,k}(z)=\langle\,\frac{\pr}{\pr z_j}+i\frac{\pr\varphi}{\pr z_j}\frac{\pr}{\pr\theta}\,|\,\frac{\pr}{\pr z_k}+i\frac{\pr\varphi}{\pr z_k}\frac{\pr}{\pr\theta}\,\rangle_{\mathcal{L}},\ \ j,k=1,\ldots,n.
\end{equation}
Put\begin{equation}\label{e-gue160601III}
h=h(z)=\left(h_{j,k}\right)^n_{j,k=1},
%\ \ h_{j,k}=\omega_{k,j},
\ \ j, k=1,\ldots,n,
\end{equation}and
$h^{-1}=\left(h^{j,k}\right)^n_{j,k=1}$, where $h^{-1}$ is the inverse matrix of $h$.
%$h^{-1}=\left(h^{j,k}\right)^n_{j,k=1}$, $h^{-1}$ is the inverse matrix of $h$.
Furthermore, set
\[\theta=\theta(z)=\partial_b h\cdot h^{-1}%h^{-1}\pr_bh
=\left(\theta_{j,k}(z)\right)^n_{j,k=1},\ \ \theta_{j,k}(z)\in T^{*1,0}_xX,\ \ x=(z,\theta),\ \ j,k=1,\ldots,n.\]
Then, \begin{equation}\label{e-gue160602I}
\begin{split} \nabla^{T^{1,0}X}_{\mathcal{L}}:C^\infty(X,T^{1,0}X)&\To C^\infty(X,\Complex T^*X\otimes T^{1,0}X)\\U=\sum^n_{j=1}a_jZ_j&\To\sum^n_{j=1}da_j\otimes Z_j+\sum^n_{j,k=1}a_j\theta_{j,k}\otimes Z_k,\ \ a_j\in C^\infty(D),\ 1\leq j\leq n.
\end{split}
\end{equation}
It is straightforward to check that the definition of $\nabla^{T^{1,0}X}_{\mathcal{L}}$ is independent of the choice of BRT trivialization and hence it is globally defined. The rigid Chern curvature with respect to $\langle\,\cdot\,|\,\cdot\,\rangle_{\mathcal{L}}$ is given by \begin{equation} \label{s1-e13}
\begin{split} &R^{T^{1,0}X}_{\mathcal{L}}=\ddbar_b\theta=\left(\ddbar_b\theta_{j,k}\right)^n_{j,k=1}=\left(\mathcal{R}_{j,k}\right)^n_{j,k=1}\in C^\infty(X,T^{*1,1}X\otimes{\rm End\,}(T^{1,0}X)),\\ &R^{T^{1,0}X}_{\mathcal{L}}(\ol U, V)\in {\rm End\,}(T^{1,0}X),\ \ \forall U, V\in T^{1,0}X,\\ &R^{T^{1,0}X}_{\mathcal{L}}(\ol U,V)\xi=\sum^n_{j,k=1}\langle\,\mathcal{R}_{j,k}\,,\ol U\wedge V\,\rangle\xi_kZ_j,\ \ \xi=\sum^n_{j=1}\xi_jZ_j,\ \ U, V\in T^{1,0}X.
\end{split}
\end{equation}
It is straightforward to check that the definition of $R^{T^{1,0}X}_{\mathcal{L}}$ is independent of the choice of BRT trivialization and hence it is globally defined.

%Set
%\begin{equation} \label{s1-e14}
%\abs{R^{T^{1,0}X}_{\mathcal{L}}}^2_{\mathcal{L}}:=\sum^n_{j,k,s,t=1}\abs{\langle\,R^{T^{1,0}X}_{\mathcal{L}}(\ol e_j,e_k)e_s\,|\,e_t\,\rangle_{\mathcal{L}}}^2,
%\end{equation}
%where $e_1,\ldots,e_n$ is an orthonormal frame for $T^{1,0}X$ with respect to $\langle\,\cdot\,|\,\cdot\,\rangle_{\mathcal{L}}$.
%It is straightforward to see that the definition of $\abs{R^{T^{1,0}X}_{\mathcal{L}}}^2_{\mathcal{L}}$ is independent of the choices of orthonormal frames. Thus, $\abs{R^{T^{1,0}X}_{\mathcal{L}}}^2_{\mathcal{L}}$ is globally defined.

%The rigid Ricci curvature with respect to $\langle\,\cdot\,|\,\cdot\,\rangle_{\mathcal{L}}$ is a global $(1,1)$ form on $X$ given by
%\begin{equation} \label{s1-e15}
%{\rm Ric\,}_{\mathcal{L}}:=-\sum^n_{j=1}\langle\,R^{TX}_\omega(\cdot,e_j)\cdot\,,e_j\,\rangle_\omega,
%\end{equation}
%\begin{equation}\label{e-gue160602a}
%\langle\,{\rm Ric\,}_{\mathcal{L}}\,, \ol U\wedge V\,\rangle=-\sum^n_{j=1}\langle\,R^{T^{1,0}X}_{\mathcal{L}}(\ol U,e_j)V\,|\, e_j\,\rangle_{\mathcal{L}},\ \
%U, V\in T^{1,0}X,
%\end{equation}
%where $e_1,\ldots,e_n$ is an orthonormal frame for $T^{1,0}X$ with respect to $\langle\,\cdot\,|\,\cdot\,\rangle_{\mathcal{L}}$. It is easy to see that ${\rm Ric\,}_{\mathcal{L}}$ is well-defined.

%\subsection{Tanaka-Webster scalar curvature}\label{s-gue160603}
\subsection{Pseudohermitian geometry}\label{s-gue160603}
In this section we will prove that the the rigid scalar curvature $S_{\mathcal{L}}$ given by \eqref{e-gue160531I} and the rigid Chern curvature given by (\ref{s1-e13}) are just
Tanaka-Webster scalar curvature and Tanaka-Webster curvature respectively up to some constants. We recall Tanaka-Webster curvature and scalar curvature first.
For this moment, we do not assume that $(X,T^{1,0}X)$ has a transversal CR $S^1$ action, that is,  we only assume that $(X,T^{1,0}X)$ is a general orientable strongly pseudoconvex CR manifold of dimension $2n+1$, $n\geq1$. Let
\[HX:=\set{U\in TX;\, U=W+\ol W,\ \ \mbox{for some $W\in T^{1,0}X$}}.\]
Then $HX$ is a subbundle of $TX$ of dimension $2n$. Let $J:HX\To HX$ be the complex structure map given by
\[\begin{split}
J:HX&\To HX,\\
U=W+\ol W, W\in T^{1,0}X,&\To\sqrt{-1}W-\sqrt{-1}~\overline{W}.
\end{split}\]
Since $X$ is orientable, there is a $\theta_0\in C^\infty(X,T^*X)$ which annihilates exactly $HX$. Any such $\theta_0$ is called a pseudohermitian structure on $X$. Then there is a unique vector field $T\in C^\infty(X,TX)$ on $X$ such that
\[\theta_0(T)\equiv1,\ \ d\theta_0(T,\cdot)\equiv0.\]
%Any such $θ_0$ is called a pseudohermitian structure on $X$.
The following is well-known

\begin{prop}[Proposition 3.1 in \cite{T75} ]\label{p-160327}
With the notations above, there is a unique linear connection (Tanaka-Webster connection) denoted by $\nabla: C^\infty(X,TX)\rightarrow C^\infty(X,T^\ast X\otimes TX)$ satisfying the following conditions:
\begin{enumerate}
  \item The contact structure $HX$ is parallel, i.e., $\nabla_U C^\infty(X,HX)\subset C^\infty(X,HX)$ for $U\in C^\infty(X,TX).$
  \item The tensor fields $T, J, d\theta_0$ are all parallel, i.e., $\nabla T=0, \nabla J=0, \nabla d\theta_0=0.$
  \item The torsion $\tau$ of $\nabla$ satisfies:
  ${\tau}(U, V)=d\theta_0(U, V)T$, ${\tau}(T, JU)=-J{\tau}(T, U)$,
  $U, V\in C^\infty(X, HX).$
\end{enumerate}
\end{prop}

Recall that $\nabla J\in C^\infty(X,T^\ast X\otimes{\rm End\,}HX)$,
$\nabla d\theta_0\in C^\infty(X,T^\ast X\otimes\Lambda^2(\Complex T^\ast X))$
are defined by $(\nabla_UJ)W=\nabla_U(JW)-J\nabla_UW$, where $U\in C^\infty(X,TX)$, $W\in C^\infty(X,HX)$, and
$\nabla_Ud\theta_0(W, V)=Ud\theta_0(W, V)-d\theta_0(\nabla_UW, V)-
d\theta_0(W, \nabla_UV)$ for $U, V, W\in C^\infty(X,TX)$.
%Similarly, for any $u\in\Omega^{0,q}(X)$, we can define
%$\nabla u\in C^\infty(T^\ast X\otimes\Lambda^q(\Complex T^\ast X))$
%in the standard way.
By (a) and $\nabla J=0$ in (b), we have
$\nabla_U C^\infty(X,T^{1, 0}X)\subset C^\infty(X,T^{1, 0}X)$ and
$\nabla_U C^\infty(X,T^{0, 1}X)\subset C^\infty(X,T^{0, 1}X)$ for
$U\in C^\infty(X,TX).$ Moreover, $\nabla J=0$ and $\nabla d\theta_0=0$
imply that the Tanaka-Webster connection is compatible with the Webster metric. That is,
\begin{equation}\label{e-gue160922}
Ud\theta_0(W, V)=d\theta_0(\nabla_UW, V)+
d\theta_0(W, \nabla_UV),\ \ \forall U, V, W\in C^\infty(X,TX).
\end{equation}
By definition, the torsion of $\nabla$ is given by
$\tau (W, U)=\nabla_WU-\nabla_UW-[W, U]$ for $U, V\in C^\infty(X,TX)$ and
$\tau(T, U)$ for $U\in C^\infty(X,HX)$ is called pseudohermitian torsion.

Let $\{Z_\alpha\}_{\alpha=1}^{n}$ be a local  frame of $T^{1, 0}X$  and let $\set{\theta^\alpha}_{\alpha=1}^{n}$ be the dual frame of $\{Z_\alpha\}_{\alpha=1}^{n}$. Write $Z_{\overline\alpha}=\overline{Z_\alpha}$, $\theta^{\overline\alpha}=\overline{\theta^\alpha}.$ Write
\[\nabla Z_\alpha=\omega_{\alpha}^\beta \otimes Z_\beta,\ \ \nabla Z_{\overline\alpha}=\omega_{\overline\alpha}^{\overline\beta} \otimes Z_{\overline\beta},\ \ \nabla T=0.\]
We call $\omega^{\beta}_{\alpha}$ the connection form of Tanaka-Webster connection with respect to the frame $\{Z_\alpha\}_{\alpha=1}^{n}.$ We denote by $\Theta_{\alpha}^\beta$ the Tanaka-Webster curvature form. We have $\Theta_{\alpha}^\beta=d\omega_{\alpha}^\beta-\omega_{\alpha}^\gamma\wedge\omega_\gamma^\beta.$ It is easy to check that
\begin{equation}\label{16-10-04-2}
\Theta_{\alpha}^\beta=R_{\alpha\ j\overline k}^{\ \ \beta}\theta^j\wedge\theta^{\overline k}+A_{\alpha\ j k}^{\ \ \beta}\theta^j\wedge\theta^{ k}+B_{\alpha\ j k}^{\ \ \beta}\theta^{\ol j}\wedge\theta^{\ol{k}}+C\wedge\theta_0,\ \ \mbox{$C$ is a one form}.
\end{equation}
We call $R_{\alpha\ j\overline k}^{\ \ \beta}$ the pseudohermitian curvature tensor and its trace
\begin{equation}\label{e-gue160923}
R_{\alpha\overline k}:=\sum_{j=1}^nR_{\alpha\ j\overline k}^{\ \ j}
\end{equation}
is called pseudohermitian Ricci tensor.  Write $d\theta_0=ig_{\alpha\overline\beta}\theta^\alpha\wedge\theta^{\overline\beta}$.
Let $\{g^{\overline\sigma\beta}\}$ be the inverse matrix of $\{g_{\alpha\overline\beta}\}$. The Tanaka-Webster Scalar curvature $R$ with respect to the pseudohermitian structure$\theta_0$ is given by
\begin{equation}\label{e-gue160923I}
R=g^{\overline k\alpha}R_{\alpha\overline k}.
\end{equation}

\begin{thm}\label{t-gue161012r}
Let $R(x)\in C^\infty(X)$ denotes Tanaka-Webster Scalar curvaturethe  on $X$ with respect to the pseudohermitian structure $\theta_0=-d\omega_0$. Then, $S_{\mathcal L}(x)=4\pi R(x)$, for every $x\in X$, where $S_{\mathcal{L}}$ denotes the rigid Scalar curvature given by \eqref{e-gue160531Im}.
\end{thm}

%Let us come back to our situation.
\begin{proof}
Let $\{Z_\alpha\}_{\alpha=1}^{n}$ be the local frame of $T^{1, 0}X$  given by some BRT trivialization (see Theorem \ref{t-gue150514}). Then $\{dz_\alpha\}_{\alpha=1}^{n}$ is a dual frame of $\{Z_\alpha\}_{\alpha=1}^{n}$. Write $Z_{\overline\alpha}=\overline{Z_\alpha}$, $\theta^\alpha=dz_{\alpha}$, $\theta^{\overline\alpha}=\overline{\theta^\alpha}.$ From (\ref{e-gue160502b}) and \eqref{e-can}, we have $d(-\omega_0)=ig_{\alpha\overline\beta}\theta^\alpha\wedge\theta^{\overline\beta}$, where $g_{\alpha\overline\beta}=2\frac{\partial^2\varphi(z)}{\partial z_\alpha\partial\overline z_\beta}.$ As before, write
%Let $\omega^{\beta}_{\alpha}$ be the connection form of Tanaka-Webster connection with respect to the frame $\{Z_\alpha\}_{\alpha=1}^{n}.$ Thus, we have
$$\nabla Z_\alpha=\omega_{\alpha}^\beta \otimes Z_\beta, \nabla Z_{\overline\alpha}=\omega_{\overline\alpha}^{\overline\beta} \otimes Z_{\overline\beta}, \nabla T=0.$$
From \eqref{e-gue160922} and by direct calculation we find
$ \omega_{\alpha}^\beta=g^{\overline\sigma\beta}\partial_bg_{\alpha\overline\sigma} $, where $\{g^{\overline\sigma\beta}\}$ is the inverse matrix of $\{g_{\alpha\overline\beta}\}$.
Since $T$ is an infinitesimal CR automorphism, Webster \cite{W78} showed that the pseudohermitian torsion vanishes on $X$, thus (\ref{16-10-04-2}) becomes
\begin{equation}
\Theta_{\alpha}^\beta=R_{\alpha\ j\overline k}^{\ \ \beta}\theta^j\wedge\theta^{\overline k}.
\end{equation}
%We denote by $\Theta_{\alpha}^\beta$ the Tanaka-Webster curvature form.
%Since the pseudohermitian torsion vanishes, %we have $\Theta_{\alpha}^\beta=d\omega_{\alpha}^\beta-\omega_{\alpha}^\gamma\wedge\omega_\gamma^\beta.$
%it is easy to check that
%$\Theta_{\alpha}^\beta=R_{\alpha\ j\overline k}^{\ \ \beta}\theta^j\wedge\theta^{\overline k}$,
%where $R_{\alpha\ j\overline k}^{\ \ \beta}$ is the Tanaka-Webster curvature
Again, a direct calculation leads to $h_{j, k}=\frac{1}{2\pi}g_{j\overline k}$ and
\begin{equation}
R_{\alpha\ j\overline k}^{\ \ \beta}\theta^j\wedge\theta^{\overline k}=-2g^{\overline\sigma\beta}\frac{\partial^4\varphi(z)}{\partial z_\alpha\partial\overline z_\sigma\partial z_j\partial\overline z_k}dz_j\wedge d\overline z_k-\frac{\partial g_{\alpha\overline\sigma}}{\partial z_j}\frac{\partial g^{\overline\sigma\beta}}{\partial \overline z_k}dz_j\wedge d\overline z_k.
\end{equation}
Then we get \begin{equation}\label{16-10-04-1}
\theta_{j, k}=\partial_b h_{j, l}\cdot h^{l, k}=\partial_b g_{j \overline l}\cdot g^{\overline l k}=\omega_{j}^k.
\end{equation}
Thus, from (\ref{s1-e13}) and (\ref{16-10-04-1}) we have  $$\mathcal R_{j, k}=\overline\partial_b{\theta_{j, k}}=\overline\partial_b(\partial_b g_{j \overline l}\cdot g^{\overline l k})=g^{\overline l k}\overline\partial_b\partial_b g_{j\overline l}-\partial_b g_{j\overline l}\wedge \overline\partial_b g^{\overline lk}=R_{j\ \alpha\overline \beta}^{\ \ k}\theta^\alpha\wedge\theta^{\overline \beta}=\Theta_{j}^k.$$
Now given any point $x_0\in X$ we can choose coordinates $(z_1, \cdots, z_n, \theta)$ in the BRT trivialization such that
\begin{equation}\label{normal1}
g_{\alpha\overline\beta}(x_0)=\lambda_\alpha\delta_{\alpha}^\beta, ~~
\frac{\partial g_{\alpha\overline\sigma}}{\partial z_j}(x_0)=\frac{\partial g_{\sigma\overline\beta}}{\partial \overline z_k}(x_0)=0
\end{equation}
holds. Thus, at $x_0$ one has $R_{\alpha\ j\overline k}^{\ \ \beta}(x_0)=(-2)\frac{1}{\lambda_{\beta}}\frac{\partial^4\varphi(z)}{\partial z_\alpha\partial\overline z_\beta\partial z_j\partial\overline z_k}\Big|_{x_0}$. The pseudohermitian Ricci tensor is given by \begin{equation}\label{Rcci1}
R_{\alpha\overline k}(x_0)=\sum_{j=1}^nR_{\alpha\ j\overline k}^{\ \ j}(x_0) =(-2)\sum_{j=1}^n\frac{1}{\lambda_{j}}\frac{\partial^4\varphi(z)}{\partial z_\alpha\partial\overline z_j\partial z_j\partial\overline z_k}\Big|_{x_0}.
\end{equation}
From (\ref{Rcci1}), the Tanaka-Webster scalar curvature is given by
\begin{equation}\label{cal3}
R(x_0)=g^{\overline k\alpha}(x_0)R_{\alpha\overline k}(x_0)=(-2)\sum_{j, k=1}^n\frac{1}{\lambda_j}\frac{1}{\lambda_k}\frac{\partial^4\varphi(z)}{\partial z_k\partial\overline z_j\partial z_j\partial\overline z_k}\Big|_{x_0}.
\end{equation}
From (\ref{e-gue160531}), by direct calculation, one finds $a(x)=\frac{1}{\pi^n}\det{g_{\alpha\overline\beta}}$ on $D$ and $a(x_0)=\frac{\lambda_1\cdots\lambda_n}{\pi^n}.$ By definition,
\begin{equation}
S_{\mathcal{L}}(x_0)=\triangle_{\mathcal{L}}(\log a(x))\Big|_{x_0}=(-2)(2\pi)g^{\overline kj}\frac{\partial^2}{\partial z_j\partial\overline z_k}\log a(x)\Big|_{x_0}=(-2)(2\pi)\sum_{j=1}^n\frac{1}{\lambda_j}
\frac{\partial^2}{\partial z_j\partial\overline z_j}\log a(x)\Big|_{x_0}.
\end{equation}
From (\ref{normal1}) and by direct calculation, we have
\begin{equation}\label{cal1}
\frac{\partial^2}{\partial z_j\partial\overline z_j}\log a(x)\Big|_{x_0}
%=\frac{\partial^2}{\partial z_j\partial\overline z_j}\log\left(1+\frac{a(x)-a(x_0)}{a(x_0)}\right)\Big|_{x_0}
=\frac{1}{a(x_0)}\frac{\partial^2}{\partial z_j\partial\overline z_j}a(x)\Big|_{x_0}
\end{equation}
and
\begin{equation}\label{cal2}
\begin{split}
\frac{\partial^2}{\partial z_j\partial\overline z_j}a(x)\Big|_{x_0}
&=\frac{1}{\pi^n}\frac{\partial^2\det{g_{\alpha\overline\beta}}}{\partial z_j\partial\overline z_j}\Big|_{x_0}=\frac{1}{\pi^n}\sum_{k=1}^n\frac{\partial^2 g_{k\overline k}}{\partial z_j\partial\overline z_j}\frac{\det{g_{\alpha\overline\beta}}}{g_{k\overline k}}\Big|_{x_0}\\
&=\frac{2}{\pi^n}\sum_{k=1}^n\frac{\partial^4\varphi}
{\partial z_j\partial\overline z_j\partial z_k\partial \overline z_k}(x_0)\frac{\lambda_1\cdots\lambda_n}{\lambda_k}.
\end{split}
\end{equation}
Thus from (\ref{cal1}) and (\ref{cal2}), we have
\begin{equation}\label{cal4}
S_{\mathcal{L}}(x_0)=-8\pi\sum_{j, k=1}^n\frac{1}{\lambda_j}\frac{1}{\lambda_k}\frac{\partial^4\varphi(z)}{\partial z_k\partial\overline z_j\partial z_j\partial\overline z_k}\Big|_{x_0}.
\end{equation}
Combining (\ref{cal3}) and (\ref{cal4}), we conclude
\begin{equation}
S_{\mathcal L}(x_0)=4\pi R(x_0).
\end{equation}
\end{proof}

\section{Szeg\H{o} kernel asymptotic expansion}\label{s-gue160822}

In this section, we will prove Theorem~\ref{t-gue160605} and Theorem~\ref{t-gue160605I}. The section is organized as follows.
We start by recalling and proving some basic facts about the Kohn Laplacian and its restriction to equivariant functions (see Section~\ref{s-gue160822I}). In Section~\ref{s-gue160830} we adapt some results on local Bergman kernel expansion to the BRT trivialized setting. In Section~\ref{s-gue160903} we use these adapted results to get an approximation for the \(m\)-th Fourier component of the Szeg\H{o} kernel. Using this approximation, we prove Theorem~\ref{t-gue160605} and Theorem~\ref{t-gue160605I} in Section~\ref{s-gue160905}.
\subsection{Kohn Laplacians}\label{s-gue160822I}

Since $T\ddbar_b=\ddbar_bT$, we have
\[\ddbar_{b,m}:=\ddbar_b:C^\infty_m(X)\To\Omega^{0,1}_m(X),\ \ \forall m\in\mathbb Z.\]
We also write
\[\ol{\pr}^{*}_b:\Omega^{0,1}(X)\To C^\infty(X)\]
to denote the formal adjoint of $\ddbar_b$ with respect to $(\,\cdot\,|\,\cdot\,)$. Since $\langle\,\cdot\,|\,\cdot\,\rangle$ is rigid, we can check that
\begin{equation}\label{e-gue150517i}
\begin{split}
&T\ddbar^{*}_b=\ddbar^{*}_bT\ \ \mbox{on $\Omega^{0,1}(X)$},\\
&\ddbar^{*}_{b,m}:=\ddbar^{*}_b:\Omega^{0,1}_m(X)\To C^\infty_m(X),\ \ \forall m\in\mathbb Z.
\end{split}
\end{equation}
The Kohn Laplacian for functions is given by
\begin{equation}\label{e-gue151113yq}
\Box_{b}:=\ddbar^*_{b}\ddbar_{b}:C^\infty(X)\To C^\infty(X).
\end{equation}

Now, we fix $m\in\mathbb Z$. The $m$-th Fourier component of Kohn Laplacian for functions is given by
\begin{equation}\label{e-gue151113y}
\Box_{b,m}:=\ddbar^*_{b,m}\ddbar_{b,m}:C^\infty_m(X)\To C^\infty_m(X).
\end{equation}
We extend $\Box_{b,m}$ to $L^{2}_{m}(X)$ by
\begin{equation}\label{e-gue151113yI}
\Box_{b,m}:{\rm Dom\,}\Box_{b,m}\subset L^{2}_m(X)\To L^{2}_m(X)\,,
\end{equation}
where ${\rm Dom\,}\Box_{b,m}:=\{u\in L^{2}_m(X);\, \Box_{b,m}u\in L^{2}_m(X)\}$ and for any $u\in L^{2}_m(X)$, $\Box_{b,m}u$ is defined in the sense of distribution.
It is well-known that $\Box_{b,m}$ is self-adjoint, ${\rm Spec\,}\Box_{b,m}$ is a discrete subset of $[0,\infty[$ and for every $\nu\in{\rm Spec\,}\Box_{b,m}$, $\nu$ is an eigenvalue of $\Box_{b,m}$ (see Section 3 in~\cite{CHT}). The following is well-known (see Theorem 3.10 in~\cite{HH16})

\begin{thm}\label{t-gue160822a}
Let $\mu_m$ be the lowest non-zero eigenvalue of $\Box_{b,m}$. There exist constants $c_1>0, c_2>0$ not depending on $m$ such that for $m\in\mathbb N$, we have
\begin{equation}\label{e-gue160822a}
\mu_m \ge c_1m-c_2.
\end{equation}
\end{thm}

Let $S_m:L^2(X)\To{\rm Ker\,}\Box_{b,m}$ be the orthogonal projection and let $S_m(x,y)\in C^\infty(X\times X)$ be the distribution kernel of $S_m$. It is clear that $S_m(x,x)=S_m(x)$, for every $x\in X$, where $S_m(x)\in C^\infty(X)$ is given by \eqref{e-gue160605}.

For $s\in\mathbb Z$, let $H^{s}(X)$ denote the Sobolev space for functions on $X$
of order $s$ and let $\norm{\cdot}_{s}$ denote the standard Sobolev norm of order $s$ with respect to $(\,\cdot\,|\,\cdot\,)$. First, we need

\begin{lem}\label{l-gue160827}
For every $s\in\mathbb N_0$, there are constants $C_{s}>0$ and $N_{s}>0$ independent of $m$ such that
\begin{equation}\label{e-gue160827}
\norm{S_{m}u}_{2s}\leq C_{s}m^{N_{s}}\norm{u}_{-2s}\ \ ~\text{for all}~ u\in H^{-2s}(X).
\end{equation}
\end{lem}

\begin{proof}
By G\"arding inequality, for every $s\in\mathbb N_0$, it is easy to see that there are constants $C_s>0$ and $N_s>0$ independent of $m$ such that
\begin{equation}\label{e-gue160827I}
\|S_{m}u\|_{s+2}\leq C_{s}m^{N_s}\Bigr(\|(\Box_{b,m}-T^2)S_{m}u\|_{s}+
\|S_{m}u\|_{s}\Bigr),\ \ \forall ~u\in L^2(X).
\end{equation}
From \eqref{e-gue160827I}, by using induction and notice that
\[\begin{split}
&T^2S_{m}u=-m^2S_{m}u,\ \ \forall u\in L^2(X),\\
&\norm{S_{m}u}\leq\norm{u},\ \ \forall u\in L^2(X),
\end{split}\]
it is straightforward to see that for every $s\in\mathbb N_0$, there are constants $\Td C_s>0$ and $\Td N_s>0$ independent of $m$ such that
\begin{equation}\label{e-gue160827II}
\norm{S_{m}u}_{2s}\leq \Td C_{s,m}m^{\Td N_s}\norm{u},\ \ \forall u\in L^2(X).
\end{equation}
From \eqref{e-gue160827II}, we deduce that
%and by some standard argument in functional analysis, it is not difficult  to see that $S_{m}$ can be extended from $L^2(X)$ to $H^{-2s}(X)$ for every $s\in\mathbb N$ and
\begin{equation}\label{e-gue160827III}
\begin{split}
&S_m:H^{-2s}X\To L^2(X),\\
&\norm{S_{m}u}\leq \Td C_{s,m}m^{\Td N_s}\norm{u}_{-2s},\ \ \forall u\in H^{-2s}(X).
\end{split}
\end{equation}
Fix $s\in\mathbb N_0$. From \eqref{e-gue160827II} and \eqref{e-gue160827III}, we have
\[
\begin{split}
\norm{S_{m}u}_{2s}=\norm{S_{m}S_{m}u}_{2s}\leq\Td C_{s_0}m^{\Td N_{s}}\norm{S_{m}u}\leq(\Td C_{s})^2m^{2\Td N_{s}}\norm{u}_{-2s},\ \ \forall u\in H^{-2s}(X).
\end{split}
\]
The lemma follows.
\end{proof}

Let $N_{m}:L^2_m(X)\To{\rm Dom\,}\Box_{b,m}$ be the partial inverse of $\Box_{b,m}$. We have
\begin{equation}\label{e-gue160827a}
\begin{split}
&\Box_{b,m}N_{m}+S_{m}=I\ \ \mbox{on $L^2_m(X)$},\\
&N_{m}\Box_{b,m}+S_{m}=I\ \ \mbox{on ${\rm Dom\,}\Box_{b,m}$}.\\
\end{split}
\end{equation}
We need

\begin{lem}\label{l-gue160325I}
For every $s\in\mathbb N_0$, there are constants $\hat C_{s}>0$ and $\hat N_{s}>0$ independent of $m$ such that
\begin{equation}\label{e-gue160325pII}
\norm{N_{m}u}_{s+2}\leq\hat C_{s}m^{\hat N_{s}}\norm{u}_{s}~\text{for all}~ u\in H^{s}(X)\bigcap L^2_m(X).
\end{equation}
\end{lem}

\begin{proof}
We will prove \eqref{e-gue160325pII} by induction over $s\in\mathbb N_0$. By G\"arding inequality, it is easy to see that there are constants $\Td C>0$, $\Td N>0$ independent of $m$ such that
\begin{equation}\label{e-gue160325y}
\|N_{m}u\|_{2}\leq\Td Cm^{\Td N}\Bigr(\|(\Box_{b,m}-T^2)N_{m}u\|+
\|N_{m}u\|\Bigr),\ \ \forall u\in L^2_m(X).
\end{equation}
From \eqref{e-gue160827a}, we have
\begin{equation}\label{e-gue160325yI}
(\Box_{b,m}-T^2) N_{m}u=(I-S_{m})u+m^2N_{m}u.
\end{equation}
From \eqref{e-gue160325yI} and Theorem~\ref{t-gue160822a}, we see that there are  constants $\hat C>0$ and $\hat N>0$ independent of $m$ satisfying
\begin{equation}\label{e-gue160325yII}
\norm{N_{m}u}+\norm{(\Box_{b,m}-T^2)N_{m}u}\leq\hat Cm^{\hat N}\norm{u},\  \ \forall u\in L^2_m(X).
\end{equation}
From \eqref{e-gue160325yII} and \eqref{e-gue160325y}, we see that \eqref{e-gue160325pII} holds for $s=0$.

We assume that \eqref{e-gue160325pII} holds for some $s_0\in\mathbb N_0$. We are going to prove that \eqref{e-gue160325pII} holds for $s_0+1$. By G\"arding inequality, it is easy to see that there are constants $\Td C_{s_0}>0$ and $\Td N_{s_0}>0$ independent of $m$ such that
\begin{equation}\label{e-gue160325yIII}
\begin{split}
&\|N_{m}u\|_{s_0+3}\\
&\leq\Td C_{s_0}m^{\Td N^{s_0}}\Bigr(\|(\Box_{b,m}-T^2)N_{m}u\|_{s_0+1}+
\|N_{m}u\|_{s_0+1}\Bigr),\ \ \forall u\in H^{s_0+1}(X)\bigcap L^2_m(X).
\end{split}
\end{equation}
From \eqref{e-gue160827a}, we have
\begin{equation}\label{e-gue160325yIV}
(\Box_{b,m}-T^2)N_{m}u=(I-S_{m})u+m^2N_{m}u.
\end{equation}
From \eqref{e-gue160827II}, we have
\begin{equation}\label{e-gue160325yV}
\norm{S_{m}u}_{s_0+1}\leq\norm{S_{m}u}_{2(s_0+1)}\leq c_{s_0}m^{N_{s_0}}\norm{u}\leq c_{s_0}m^{N_{s_0}}\norm{u}_{s_0+1},
\end{equation}
where  $c_{s_0}>0$, $N_{s_0}$ are constants independent of $m$. By induction hypothesis, we have
\begin{equation}\label{e-gue160325yVI}
\norm{N_{m}u}_{s_0+1}\leq\norm{N_{m}u}_{s_0+2}\leq\hat c_{s_0}m^{\hat N_{s_0}}\norm{u}_{s_0}\leq\hat c_{s_0}m^{\hat N_{s_0}}\norm{u}_{s_0+1},
\end{equation}
where $\hat c_{s_0}>0$, $\hat N_{s_0}$ are constants independent of $m$. From \eqref{e-gue160325yVI}, \eqref{e-gue160325yV}, \eqref{e-gue160325yIV} and \eqref{e-gue160325yIII}, we see that \eqref{e-gue160325pII} holds for $s_0+1$. The lemma follows.
\end{proof}

\subsection{Approximate Bergman kernels on BRT
trivializations}\label{s-gue160830}

From now on, we fix $m\in\mathbb{Z}$. Let $B:=(D,(z,\theta),\varphi)$ be a
BRT trivialization. We may assume that $D=U\times]-\varepsilon,\varepsilon[$,
where $\varepsilon>0$ and $U$ is an open set of $\mathbb{C}^n$.  Consider $
L\rightarrow U$ be a trivial line bundle with non-trivial Hermitian fiber
metric $\left\vert 1\right\vert^2_{h^L}=e^{-2\varphi}$. Let $
(L^m,h^{L^m})\rightarrow U$ be the $m$-th power of $(L,h^L)$. Let $\Omega^{0,q}(U,L^m)$ be the space of $(0,q)$ forms on $U$ with
values in $L^m$, $q=0,1,2,\ldots,n$. Put $C^\infty(U,L^m):=\Omega^{0,0}(U,L^m)$.
Let $\langle\,\cdot\,,\,\cdot\,
\rangle$ be the Hermitian metric on $\mathbb{C }TU$ given by
\begin{equation*}
\langle\,\frac{\partial}{\partial z_j}\,,\,\frac{\partial}{\partial z_k}
\,\rangle=\langle\,\frac{\partial}{\partial z_j}+i\frac{\partial\varphi}{
\partial z_j}(z)\frac{\partial}{\partial\theta}\,|\,\frac{\partial}{\partial
z_k}+i\frac{\partial\varphi}{\partial z_k}(z)\frac{\partial}{\partial\theta}
\,\rangle,\ \ j,k=1,2,\ldots,n.
\end{equation*}
The Hermitian metric $\langle\,\cdot\,,\,\cdot\,\rangle$ induces Hermitian metrics on $T^{*p,q}U$
bundle of $(p,q)$ forms on $U$, $p,q=0,1,\ldots,n$, also denoted by $\langle\,\cdot\,,\,\cdot\,\rangle$. Let $dv_U$ be the volume form on $U$ induced by $\langle\,\cdot\,,\,\cdot\,\rangle$. Note that $dv_X=dv_U(x)d\theta$ on $D$. Let $(\,\cdot\,,\,\cdot\,)_m$ be the $L^2$ inner product on $\Omega^{0,q}(U,L^m)$ induced by $\langle\,\cdot\,,\,\cdot\,\rangle$ and $h^{L^m}$, $q=0,1,2,\ldots,n$. Let $\overline\partial:C^\infty(U,L^m)\rightarrow\Omega^{0,1}(U,L^m)$
be the Cauchy-Riemann operator and let $\overline{\partial}^{*,m}:\Omega^{0,1}(U,L^m)\rightarrow C^\infty(U,L^m)$
be the formal adjoint of $\overline\partial$ with respect to $(\,\cdot\,,\,\cdot\,)_m$. Put
\begin{equation}  \label{e-gue160830I}
\Box_{B,m}:=\overline{\partial}^{*,m}\overline\partial: C^\infty(U,L^m)\To C^\infty(U,L^m).
\end{equation}

We need the following result which is well-known (see Lemma 5.1 in~\cite{CHT})

\begin{lem}\label{l-gue160830}
Let $u\in C^\infty_m(X)$ be a function. On $D$, we write $u(z,\theta)=e^{im\theta}\Td u(z)$, $\Td u(z)\in C^\infty(U,L^m)$. Then,
\begin{equation}\label{e-gue160830II}
e^{-m\varphi}\Box_{B,m}(e^{m\varphi}\Td u)=e^{-im\theta}\Box_{b,m}(u).
\end{equation}
\end{lem}

Let $A:C^\infty_0(U)\To C^\infty(U)$ be a smoothing operator and let $A(z,w)\in C^\infty(U\times U)$ be the distribution kernel of $A$ with respect to $dv_U$. Note that
\[(Au)(z)=\int A(z,w)u(w)dv_U(w),\ \ \forall u\in C^\infty_0(U).\]

The following theorem is well-known (see Theorem 3.11 and Theorem 3.12 in~\cite{HM12})

\begin{thm}\label{t-gue160901}
There is a properly supported smoothing operator $P_{B,m}:C^\infty(U)\To C^\infty(U)$ such that
\begin{equation}\label{e-gue160902}
P_{B,m}\circ e^{-m\varphi}\circ\Box_{B,m}\circ e^{m\varphi}=O(m^{-\infty})\ \ \mbox{on $U\times U$},
\end{equation}
\begin{equation}\label{e-gue160902I}
e^{m\varphi}P_{B,m}e^{-m\varphi}u=u,\ \ \mbox{for all $u\in C^\infty(U)$ with $\Box_{B,m}u=0$},
\end{equation}
hold and the distribution kernel $P_{B,m}(z,w)$ of $P_{B,m}$ satisfies
\begin{equation} \label{e-gue160902II}
P_{B,m}(z,w)=e^{im\Psi_B(z,w)}b_B(z,w,m)+
O(m^{-\infty}),
\end{equation}
with
\begin{equation}  \label{e-gue160902III}
\begin{split}
&b_B(z,w,m)\in S^{n}_{{\rm loc\,}}(1;U\times U), \\
&b_B(z,w,m)\sim\sum^\infty_{j=0}b_{B,j}(z, w)m^{n-j}\text{ in }S^{n}_{{\rm loc\,}}
(1;D\times D), \\
&b_{B,j}(z, w)\in C^\infty(D\times D),\ \ j=0,1,2,\ldots,
\end{split}
\end{equation}
$\Psi_B(z,w)\in C^\infty(U\times U)$, $\Psi(z,w)=0\Leftrightarrow z=w$ and for every compact set $K\Subset U$, there is a constant $c_K>0$ such that
\begin{equation}\label{e-gue160902IV}
{\rm Im\,}\Psi_B(z,w)\geq c_K\abs{z-w}^2,\ \ \forall (z,w)\in K\times K.
\end{equation}
\end{thm}

\begin{rem}\label{r-gue160903}
It is well-known that (see page 29 in Grigis and Sj\"ostrand~\cite{GrSj94}) there is a smooth function $\chi(x,y)\in C^\infty(U\times U)$ with ${\rm Supp\,}\chi$ is proper and $\chi(x,y)-1$ vanishes to infinite order at $x=y$. We can replace $P_{B,m}(z,w)$ by $P_{B,m}(z,w)\chi(z,w)$ and we can take $b_B(z,w,m)$ and $b_{B,j}(z,w)$, $j=0,1,2,\ldots$, are all properly supported. From now on, we assume that $b_B(z,w,m)$ and $b_{B,j}(z,w)$, $j=0,1,2,\ldots$, are all properly supported.
\end{rem}

In~\cite{HM12}, we also determine $b_{B,j}(z,z)$, $j=0,1,2$. To state the results, we need to introduce some notations. Let $R^L:=2\pr\ddbar\varphi$ and let $\dot{R}^L\in C^\infty(U, {\rm End\,}(T^{1,0}U))$ be the Hermitian matrix defined by
\begin{equation} \label{s1-e3}
\langle\,\dot{R}^L(x)\mathcal{V}\,,\,\mathcal{W}\,\rangle=\langle\,R^L(x)\,,\mathcal{V}\wedge\ol{\mathcal{W}}\,\rangle.
\end{equation}
 Put
\begin{equation} \label{s1-e6}
\omega:=\frac{\sqrt{-1}}{2\pi}R^L.
\end{equation}
The real two form $\omega$ induces a Hermitian metric $\langle\,\cdot\,,\,\cdot\,\rangle_{\omega}$ on $\Complex TU$.
The Hermitian metric $\langle\,\cdot\,,\,\cdot\,\rangle_{\omega}$ on $\Complex TU$
induces a Hermitian metric on
$T^{*p,q}U\otimes T^{*r,s}U$, $p, q, r, s\in\mathbb N_0$, also denoted by $\langle\,\cdot\,,\,\cdot\,\rangle_{\omega}$.
For $u\in T^{*p,q}U$, we denote $\abs{u}^2_{\omega}:=\langle\,u,u\,\rangle_{\omega}$.

Let $\Theta$ be the real two form induced by $\langle\,\cdot\,,\,\cdot\,\rangle$.
Put
\begin{equation} \label{s1-e7}
\begin{split}
&\omega=\sqrt{-1}\sum^n_{j,k=1}\omega_{j,k}dz_j\wedge d\ol z_k,\\
&\Theta_U=\sqrt{-1}\sum^n_{j,k=1}\Theta_{j,k}dz_j\wedge d\ol z_k.
\end{split}
\end{equation}
We notice that $\Theta_{j,k}=\langle\,\frac{\pr}{\pr z_j}\,,\,\frac{\pr}{\pr z_k}\,\rangle$, $\omega_{j,k}=\langle\,\frac{\pr}{\pr z_j}\,,\,\frac{\pr}{\pr z_k}\,\rangle_{\omega}$, $j,k=1,\ldots,n$.
Put
\begin{equation} \label{s1-e8}
h=\left(h_{j,k}\right)^n_{j,k=1},\ \ h_{j,k}=\omega_{k,j},\ \ j, k=1,\ldots,n,
\end{equation}
and $h^{-1}=\left(h^{j,k}\right)^n_{j,k=1}$, $h^{-1}$ is the inverse matrix of $h$. The complex Laplacian with respect to $\omega$ is given by
\begin{equation} \label{s1-e9}
\triangle_{\omega}=(-2)\sum^n_{j,k=1}h^{j,k}\frac{\pr^2}{\pr z_j\pr\ol z_k}.
\end{equation}
We notice that $h^{j,k}=\langle\,dz_j\,,\,dz_k\,\rangle_{\omega}$, $j, k=1,\ldots,n$. Put
\begin{equation} \label{s1-e10}
\begin{split}
&V_\omega:=\det\left(\omega_{j,k}\right)^n_{j,k=1},\\
&V_{\Theta_U}:=\det\left(\Theta_{j,k}\right)^n_{j,k=1}
\end{split}
\end{equation}
and set
\begin{equation} \label{s1-e11}
\begin{split}
&r=\triangle_{\omega}\log V_\omega,\\
&\hat r=\triangle_{\omega}\log V_{\Theta_U}.
\end{split}
\end{equation}
The function $r$ is called the scalar curvature with respect to $\omega$. Let $R^{\det}_{\Theta_U}$ be the curvature of the determinant line bundle of $T^{*1,0}U$ with respect to the real two form $\Theta_U$. Recall the identity
\begin{equation} \label{s1-e12}
R^{\rm det\,}_{\Theta_U}=\pr\ddbar\log V_{\Theta_U}.
\end{equation}

Let $h$ be as in \eqref{s1-e8}. Put
$\alpha=h^{-1}\pr h=\left(\alpha_{j,k}\right)^n_{j,k=1}$, $\alpha_{j,k}\in T^{*1,0}U$, $j,k=1,\ldots,n$.
$\alpha$ is the Chern connection
matrix with respect to $\omega$. The Chern curvature with respect to $\omega$ is given by
\begin{equation} \label{s1-e13a}
\begin{split}
&R^{TU}_{\omega}=\ddbar\alpha=\left(\ddbar\alpha_{j,k}\right)^n_{j,k=1}=\left(\mathcal{R}_{j,k}\right)^n_{j,k=1}
\in C^\infty(U,T^{*1,1}U\otimes{\rm End\,}(T^{1,0}U)),\\
&R^{TU}_{\omega}(\ol{\mathcal{V}}, \mathcal{W})\in {\rm End\,}(T^{1,0}U),\ \ \forall\mathcal{V}, \mathcal{W}\in T^{1,0}U,\\
&R^{TU}_\omega(\ol{\mathcal{V}},\mathcal{W})\xi=\sum^n_{j,k=1}\langle\,\mathcal{R}_{j,k}\,,\ol{\mathcal{V}}\wedge\mathcal{W}\,\rangle\xi_k\frac{\pr}{\pr z_j},\ \ \xi=\sum^n_{j=1}\xi_j\frac{\pr}{\pr z_j},\ \ \mathcal{V}, \mathcal{W}\in T^{1,0}U.
\end{split}
\end{equation}
Set
\begin{equation} \label{s1-e14a}
\abs{R^{TU}_{\omega}}^2_{\omega}:=\sum^n_{j,k,s,t=1}\abs{\langle\,R^{TU}_{\omega}(\ol e_j,e_k)e_s\,,e_t\,\rangle_{\omega}}^2,
\end{equation}
where $e_1,\ldots,e_n$ is an orthonormal frame for $T^{1,0}U$ with respect to $\langle\,\cdot\,,\cdot\,\rangle_{\omega}$.
It is straightforward to see that the definition of $\abs{R^{TU}_{\omega}}^2_{\omega}$ is independent of the choices of orthonormal frames. Thus, $\abs{R^{TU}_{\omega}}^2_{\omega}$ is globally defined. The Ricci curvature with respect to $\omega$ is given by
\begin{equation} \label{s1-e15}
{\rm Ric\,}_{\omega}:=-\sum^n_{j=1}\langle\,R^{TU}_\omega(\cdot,e_j)\cdot\,,e_j\,\rangle_\omega,
\end{equation}
where $e_1,\ldots,e_n$ is an orthonormal frame for $T^{1,0}U$ with respect to $\langle\,\cdot\,,\,\cdot\,\rangle_{\omega}$. That is,
\[\langle\,{\rm Ric\,}_{\omega}\,, \mathcal{V}\wedge\mathcal{W}\,\rangle=-\sum^n_{j=1}\langle\,R^{TU}_\omega(\mathcal{V},e_j)\mathcal{W}\,,e_j\,\rangle_\omega,\ \
\mathcal{V}, \mathcal{W}\in\Complex TU.\]
${\rm Ric\,}_{\omega}$ is a smooth $(1,1)$ form.

The following result is well-known (see Section 4.5 in~\cite{HM12})

\begin{thm}\label{t-gue160902}
 With the notations used above, for $b_{B,0}(z,z)$, $b_{B,1}(z)$, $b_{B,2}(z)$
in \eqref{e-gue160902III}, we have
\begin{equation} \label{e-gue160902V}
b_{B,0}(z,z)=(2\pi)^{-n}\det\dot{R}^L(z),
\end{equation}
\begin{equation} \label{e-gue160902VI}
b_{B,1}(z,z)=(2\pi)^{-n}\det\dot{R}^L(z)\Bigr(\frac{1}{4\pi}\hat r-\frac{1}{8\pi} r\Bigr)(z),
\end{equation}
\begin{equation} \label{e-gue160902VII}
\begin{split}
b_{B,2}(z,z)&=(2\pi)^{-n}\det\dot{R}^L(z)\Bigr(\frac{1}{128\pi^2}r^2-\frac{1}{32\pi^2}r\hat r+\frac{1}{32\pi^2}(\hat r)^2
-\frac{1}{32\pi^2}\triangle_{\omega}\hat r
-\frac{1}{8\pi^2}\abs{R^{\det}_{\Theta_U}}^2_{\omega}\\
&\quad+\frac{1}{8\pi^2}\langle\,{\rm Ric\,}_{\omega}\,,R^{\det}_{\Theta_U}\,\rangle_{\omega}+\frac{1}{96\pi^2}\triangle_{\omega}r-
\frac{1}{24\pi^2}\abs{{\rm Ric\,}_{\omega}}^2_{\omega}+\frac{1}{96\pi^2}\abs{R^{TU}_{\omega}}^2_{\omega}\Bigr)(z).
\end{split}
\end{equation}
We remind that $\dot{R}^L$ is given by \eqref{s1-e3} and
\[\det\dot{R}^L(z)=\lambda_1(z)\cdots\lambda_n(z),\]
where $\lambda_1(z),\ldots,\lambda_n(z)$ are eigenvalues of $\dot{R}^L(z)$.
\end{thm}

The following follows from straightforward calculation. We omit the details.

\begin{lem}\label{l-gue160902}
We have
\begin{equation}\label{e-gue160902a}
\begin{split}
&\det\dot{R}^L(z)=\det\dot{\mathcal{R}}(z,\theta),\ \ \forall (z,\theta)\in D,\\
&\hat r(z)=S^\Theta_{\mathcal{L}}(z,\theta),\ \ r(z)=S_{\mathcal{L}}(z,\theta),\ \ \forall (z,\theta)\in D,\\
&\triangle_{\omega}\hat r(z)=\triangle_{\mathcal{L}}S^\Theta_{\mathcal{L}}(z,\theta),\ \ \triangle_{\omega}r(z)=\triangle_{\mathcal{L}}S_{\mathcal{L}}(z,\theta),\ \ \forall (z,\theta)\in D,\\
&\abs{R^{\det}_{\Theta_U}}^2_{\omega}(z)=\abs{R^{\det}_{\Theta}}^2_{\mathcal{L}}(z,\theta),\ \ \abs{{\rm Ric\,}_{\omega}}^2_{\omega}(z)=\abs{{\rm Ric\,}_{\mathcal{L}}}^2_{\mathcal{L}}(z,\theta),\ \ \forall (z,\theta)\in D,\\
&\langle\,{\rm Ric\,}_{\omega}\,,\,R^{\det}_{\Theta_U}\,\rangle_{\omega}(z)=\langle\,{\rm Ric\,}_{\mathcal{L}}\,|\,R^{\det}_\Theta\,\rangle_{\mathcal{L}}(z,\theta),\  \
\abs{R^{TU}_{\omega}}^2_{\omega}(z)=\abs{R^{T^{1,0}X}_{\mathcal{L}}}^2_{\mathcal{L}}(z,\theta),\ \ \forall(z,\theta)\in D,
\end{split}
\end{equation}
where $\dot{\mathcal{R}}(x)\in \operatorname{End}(T^{1,0}_xX)$ is given by \eqref{E:1.5.15m}, $\det\dot{\mathcal{R}}(x)=\lambda_1(x)\cdots\lambda_{n}(x)$, $\lambda_1(x),\ldots,\lambda_{n}(x)$ are eigenvalues of $\dot{\mathcal{R}}(x)$, $\triangle_{\mathcal{L}}$, $S^\Theta_{\mathcal{L}}$, $S_{\mathcal{L}}$,  $R^{\det}_{\Theta}$, ${\rm Ric\,}_{\mathcal{L}}$, $R^{T^{1,0}X}_{\mathcal{L}}$ are given by \eqref{s1-e9m}, \eqref{e-gue160601Im}, \eqref{e-gue160531Im}, \eqref{e-gue160601IIm}, \eqref{e-gue160602am} and \eqref{s1-e13} respectively and $\langle\,\cdot\,|\,\cdot\,\rangle_{\mathcal{L}}$ is as in the discussion after \eqref{e-gue160605aII}.
\end{lem}

From Theorem~\ref{t-gue160902} and \eqref{e-gue160902a}, we deduce

\begin{thm}\label{t-gue160902I}
With the notations used above, for $b_{B,0}(z,z), b_{B,1}(z,z), b_{B,2}(z,z)$ in \eqref{e-gue160902III}, we have
\begin{equation} \label{e-gue160902VIII}
b_{B,0}(z,z)=(2\pi)^{-n}\det\dot{\mathcal{R}}(z,\theta),\ \ \forall (z,\theta)\in D,
\end{equation}
\begin{equation} \label{e-gue160902IX}
b_{B,1}(z,z)=(2\pi)^{-n}\det\dot{\mathcal{R}}(z,\theta)\Bigr(\frac{1}{4\pi}S^\Theta_{\mathcal{L}}-\frac{1}{8\pi}S_{\mathcal{L}}\Bigr)(z,\theta),\ \ \forall(z,\theta)\in D,
\end{equation}
\begin{equation} \label{e-gue160902X}
\begin{split}
b_{B,2}(z,z)&=(2\pi)^{-n}\det\dot{\mathcal{R}}(z,\theta)\Bigr(\frac{1}{128\pi^2}(S_{\mathcal{L}})^2-\frac{1}{32\pi^2}S_{\mathcal{L}}S^\Theta_{\mathcal{L}}+\frac{1}{32\pi^2}(S^\Theta_{\mathcal{L}})^2
-\frac{1}{32\pi^2}\triangle_{\mathcal{L}}S^\Theta_{\mathcal{L}}
-\frac{1}{8\pi^2}\abs{R^{\det}_\Theta}^2_{\mathcal{L}}\\
&\quad+\frac{1}{8\pi^2}\langle\,{\rm Ric\,}_{\mathcal{L}}\,|\,R^{\det}_\Theta\,\rangle_{\mathcal{L}}+\frac{1}{96\pi^2}\triangle_{\mathcal{L}}S_{\mathcal{L}}-
\frac{1}{24\pi^2}\abs{{\rm Ric\,}_{\mathcal{L}}}^2_{\mathcal{L}}+\frac{1}{96\pi^2}\abs{R^{T^{1,0}X}_{\mathcal{L}}}^2_{\mathcal{L}}\Bigr)(z,\theta),\  \forall(z,\theta)\in D.
\end{split}
\end{equation}
\end{thm}

\subsection{Approximate Szeg\H{o} kernels}\label{s-gue160903}

Assume that $X=D_1\bigcup D_2\bigcup\cdots\bigcup D_N$, where $B_j:=(D_j,(z,\theta),\varphi_j)$ is a
BRT trivialization, for each $j$. We may assume that for each $j$, $
D_j=U_j\times]-4\delta,4\delta[\subset\mathbb{C}^n\times
\mathbb{R}$, $U_j=\left\{z\in\mathbb{C}
^n;\, \left\vert z\right\vert<\gamma_j\right\}$ and $\delta>0$ is a small constant satisfying
\begin{equation}  \label{e-gue160903}
0<\delta<\inf\left\{\frac{\pi}{p_t},\left\vert \frac{2\pi}{p_r}-\frac{2\pi}{
p_{r+1}}\right\vert, r=1,\ldots,t-1\right\}.
\end{equation}
For each $j$, put $\hat
D_j=\hat U_j\times]-\frac{\delta}{2},\frac{\delta}{2}[$, where
$\hat U_j=\left\{z\in\mathbb{C}^n;\, \left\vert z\right\vert<\frac{\gamma_j}{
2}\right\}$. We may suppose that $X=\hat D_1\bigcup\hat
D_2\bigcup\cdots\bigcup\hat D_N$. Let $\chi_j\in C^\infty_0(\hat D_j)$, $
j=1,2,\ldots,N$, with $\sum^N_{j=1}\chi_j=1$ on $X$. Fix $j=1,2,\ldots,N$ and choose $\sigma_j\in C^\infty_0(]-\frac{\delta}{2}
,\frac{\delta}{2}[)$ with $\int\sigma_j(\theta)d\theta=1$. Let $
P_{B_j,m}(z,w)\in C^\infty(U_j\times U_j)$ be as in Theorem~\ref
{t-gue160901}. Put
\begin{equation}  \label{e-gue150627f}
\begin{split}
H_{j,m}(x,y)=\chi_j(x)e^{im\theta}P_{B_j,m}(z,w)e^{-im\eta}\sigma_j(\eta),
%&\widetilde H_{j,m}(x,y)=\chi_j(x)e^{im\theta}P_{B_j,m}(z,w)e^{-im\eta},
\end{split}
\end{equation}
where $x=(z,\theta)$, $y=(w,\eta)\in\mathbb{C}^{n}\times\mathbb{R}$. Let $
H_{j,m}$ be the continuous operator
\begin{equation}  \label{e-gue150627fI}
\begin{split}
H_{j,m}:C^\infty(X)&\rightarrow C^\infty(X), \\
u&\rightarrow\int
\chi_j(x)e^{im\theta}P_{B_j,m}(z,w)e^{-im\eta}\sigma_j(\eta)u(y)dv_X(y).
\end{split}
\end{equation}
Since $P_{B_j,m}$ is properly supported, \eqref{e-gue150627fI} is well-defined.

%We pause and introduce some notations. Let $M$ be a $C^\infty$ orientable paracompact
%manifold. We need

%\begin{defn} \label{d-gue150608}
%Let $a_j(x,m)\in C^\infty(M)$ be $m$-dependent smooth function, $m\in\mathbb N$, with
%\[\norm{a_j(x,m)}_{C^\ell(M)}\leq C_\ell m^\ell,\]
%for every $m\in\mathbb N$, $\ell\in\mathbb N_0$, where $C_\ell>0$ is a constant independent of $m$, $ j=0,1,2,\ldots$.
%Let $A(x,m)\in C^\infty(M)$ be $m$-dependent smooth function, $m\in\mathbb N$. We
%write
%\begin{equation*}
%A(x,m)\sim m^ka_{0}(x,m)+m^{k-1}a_{1}(x,m)+m^{k-2}a_{k-2}(x,m)+\cdots\
%\ \mbox{as $m\To+\infty$},
%\end{equation*}
%where $k\in\mathbb Z$, if for every $\ell\in\mathbb N_0$, every compact set $K\Subset M$ and every $N_0\in\mathbb N$, there are $C_{K,N_0,\ell}>0$ and $L_0>0
%$ independent of $m$ such that
%\begin{equation*}
%\norm{A(x,m)-\sum^{N_0}_{j=0}m^{k-j}a_{j}(x,m)}_{C^\ell(K)}\leq
%C_{K,N_0,\ell}m^{k-N_0-1+\ell},\ \ \forall m\geq L_0.
%\end{equation*}
%\end{defn}

Let $Q_m:L^2(X)\To L^2_m(X)$ be the orthogonal projection with respect to $(\,\cdot\,|\,\cdot\,)$.
Consider
\begin{equation}  \label{e-gue150627fII}
\Gamma_m:=\sum^N_{j=1}H_{j,m}\circ
Q_m:C^\infty(X)\rightarrow C^\infty(X)
\end{equation}
and let $\Gamma_m(x,y)\in C^\infty(X\times X)$ be the distribution
kernel of $\Gamma_m$. Let $\Psi_{B_j}(z,w)$, $b_{B_j}(z,w,m)$, $b_{B_j,k}(z,w)$, $k=0,1,2,\ldots$, be as in Theorem~\ref{t-gue160901}. Let $
\hat\sigma_j\in C^\infty_0(]-\delta,\delta[)$ such that $
\hat\sigma_j=1$ on some neighbourhood of $\mathrm{Supp\,}\sigma_j$ and $
\hat\sigma_j(\theta)=1$ if $(z,\theta)\in\mathrm{Supp\,}\chi_j$. Put
\begin{equation}  \label{e-gue150901}
\begin{split}
&\hat\Psi_{B_j}(x,y)=\hat\sigma_j(\theta)\Psi_{B_j}(z,w)\hat\sigma_j(\eta)\in
C^\infty(D_j\times D_j),\ \ x=(z,\theta),\ \ y=(w,\eta), \\
&\hat b_{B_j,m}(x,y)=\chi_j(x)e^{im\theta}b_{B_j,m}(t,z,w)e^{-im\eta}\sigma_j(\eta), \\
&\hat b_{B_j,k}(x,y)=\chi_j(x)e^{im\theta}b_{B_j,k}(z,w)e^{-im\eta}\sigma_j(\eta),\ \ k=0,1,2,\ldots.
\end{split}
\end{equation}
It is not difficult to check that
\begin{equation}  \label{e-gue150626fIII}
\begin{split}
\Gamma_m(x,y)&=\frac{1}{2\pi}\sum^N_{j=1}\int^\pi_{-\pi}H_{j,m}(x,e^{iu}\circ
y)e^{imu}du \\
&=\frac{1}{2\pi}\sum^N_{j=1}\int^\pi_{-\pi}e^{im\hat
\Psi_{B_j}(x,e^{iu}\circ y)}\hat b_{B_j,m}(x,e^{iu}\circ y)e^{imu}du.
%&\sim m^nb_0(x,y,m)+m^{n-1}b_1(x,y,m)+\cdots\ \ \mbox{as $m\To+\infty$},
\end{split}
\end{equation}
%where
%\begin{equation}  \label{e-gue150901a}
%\begin{split}
%&b_k(x,y,m) \\
%&=\frac{1}{2\pi}\sum^N_{j=1}\int^\pi_{-\pi}e^{im\hat\Psi_{B_j}(x,e^{iu}\circ y)}\hat b_{B_j,k}(x,e^{iu}\circ y)e^{imu}du,\ \
%k=0,1,2,\ldots.
%\end{split}
%\end{equation}

\begin{lem} \label{l-gue150626f}
We have $\Gamma_mS_m=S_m$ on $C^\infty(X)$.
\end{lem}

\begin{proof}
Let $f\in C^\infty(X)$ be a function and set $S_mf=u$. Then one has $u\in{\rm Ker\,}\Box_{b,m}\bigcap C^\infty_m(X)$. On $D_j$, set $
Q_mu=e^{im\eta}v_j(w)$, $v_j(w)\in C^\infty(U_j)$. From Lemma~\ref{l-gue160830}, we see that $\Box_{B_j,m}(e^{m\varphi_j}v_j)=0$ holds. From this observation and \eqref{e-gue160902I}, we have
\begin{equation}\label{e-gue160904}
P_{B_j,m}v_j=v_j\ \ \mbox{on $U_j$}.
\end{equation}
From \eqref{e-gue160904}, \eqref{e-gue150627fI} and since $P_{B_j,m}$ is properly supported, we have
\begin{equation*}
\begin{split}
H_{j,m}Q_mu&=\int\chi_j(x)e^{im
\theta}P_{B_j,m}(z,w)e^{-im\eta}\sigma_j(\eta)e^{im
\eta}v_j(w)dv_{U_j}(w)d\eta \\
&=\int\chi_j(x)e^{im
\theta}P_{B_j,m}(z,w)v_j(w)dv_{U_j}(w)=\chi_j(x)e^{im\theta}v_j(z)=\chi_ju.
\end{split}
\end{equation*}
Thus, $\Gamma_mu=\sum^N_{j=1}H_{j,m}
\circ Q_mu=\sum^N_{j=1}\chi_ju=u$.
The lemma follows.
\end{proof}

\begin{lem} \label{l-gue160904a}
We have $\Gamma_m\Box_b=O(m^{-\infty})$ on $X\times X$.
\end{lem}

\begin{proof}
Let $f\in C^\infty(X)$ be a function. Fix $j=1,2,\ldots,N$. On $D_j$, set $
Q_mf=e^{im\eta}v_j(w)$, $v_j(w)\in C^\infty(U_j)$. From \eqref{e-gue150627fI} and Lemma~\ref{l-gue160830}, we have
\begin{equation}\label{e-gue160904aI}
\begin{split}
&H_{j,m}Q_m\Box_bf=H_{j,m}\Box_{b,m}Q_mf\\
&=\int\chi_j(x)e^{im\theta}P_{B_j,m}(z,w)e^{-im\eta}\sigma_j(\eta)(\Box_{b,m}Q_mf)(w,\eta)dv_{U_j}(w)d\eta \\
&=\int\chi_j(x)e^{im
\theta}P_{B_j,m}(z,w)\sigma_j(\eta)e^{-m\varphi_j(w)}\Box_{B_j,m}(e^{m\varphi_j}v_j)(w)dv_{U_j}(w)d\eta\\
&=\int\chi_j(x)e^{im
\theta}P_{B_j,m}(z,w)e^{-m\varphi_j(w)}\Box_{B_j,m}(e^{m\varphi_j}v_j)(w)dv_{U_j}(w)\\
&=\int\chi_j(x)e^{im
\theta}(P_{B_j,m}\circ e^{-m\varphi_j}\circ\Box_{B_j,m}\circ e^{m\varphi_j})(z,w)v_j(w)dv_{U_j}(w)\\
&=\int\chi_j(x)e^{im
\theta}(P_{B_j,m}\circ e^{-m\varphi_j}\circ\Box_{B_j,m}\circ e^{m\varphi_j})(z,w)e^{-im\eta}\sigma_j(\eta)(Q_mf)dv_{U_j}(w)d\eta\\
&=\int K_m(x,y)(Q_mf)dv_{U_j}(w)d\eta,
\end{split}
\end{equation}
where $K_m(x,y):=\chi_j(x)e^{im
\theta}(P_{B_j,m}\circ e^{-m\varphi_j}\circ\Box_{B_j,m}\circ e^{m\varphi_j})(z,w)e^{-im\eta}\sigma_j(\eta)$. In view of \eqref{e-gue160902}, we know that
\begin{equation}\label{e-gue160904b}
K_m(x,y)=O(m^{-\infty})\ \ \mbox{on $D_j\times D_j$}.
\end{equation}
Note that $K_m(x,y)\in C^\infty_0(D_j\times D_j)$.
From \eqref{e-gue160904aI}, we see that the distribution kernel of $H_{j,m}Q_m\Box_b$ is given by $(K_m\circ Q_m)(x,y)$. From \eqref{e-gue160904b}, it is easy to see that
$(K_m\circ Q_m)(x,y)=O(m^{-\infty})$ on $D_j\times D_j$. Thus, we have $H_{j,m}Q_m\Box_b=O(m^{-\infty})$ on $D_j\times D_j$ and hence  $\Gamma_m\Box_b=O(m^{-\infty})$ on $X\times X$.
The lemma follows.
\end{proof}

Let $\Gamma^*_m:C^\infty(X)\To C^\infty_m(X)$ be the adjoint of $\Gamma_m$ with respect to $(\,\cdot\,|\,\cdot\,)$. From Lemma~\ref{l-gue150626f} and Lemma~\ref{l-gue160904a}, we deduce

\begin{thm}\label{t-gue160904f}
With the notations above, we have
\begin{equation}\label{e-gue160904w}
S_m\Gamma^*_m=S_m\ \ \mbox{on $C^\infty(X)$}
\end{equation}
and
\begin{equation}\label{e-gue160904wI}
\Box_b\Gamma^*_m=O(m^{-\infty})\ \ \mbox{on $X\times X$}.
\end{equation}
\end{thm}

Now we can prove

\begin{thm}\label{t-gue160904fI}
With the notations used above, we have $\Gamma_m=S_m+O(m^{-\infty})$ on $X\times X$.
\end{thm}

\begin{proof}
From \eqref{e-gue160827a}, we have
\begin{equation}\label{e-gue160904t}
N_m\Box_bQ_m+S_mQ_m=Q_m\ \ \mbox{on $C^\infty(X)$}.
\end{equation}
From \eqref{e-gue160904t} and \eqref{e-gue160904w}, we have
\begin{equation}\label{e-gue160905}
\begin{split}
\Gamma^*_m&=Q_m\Gamma^*_m=(N_m\Box_bQ_m+S_m)\Gamma^*_m\\
&=N_m\Box_bQ_m\Gamma^*_m+S_m\Gamma^*_m\\
&=N_m\Box_b\Gamma^*_m+S_m.
\end{split}
\end{equation}
From \eqref{e-gue160904wI} and \eqref{e-gue160325pII}, we conclude that
\[N_m\Box_b\Gamma^*_m=O(m^{-N}):H^{-s}(X)\To H^s(X)\]
holds for every $s\in\mathbb N$ and $N>0$. Hence, we find $N_m\Box_b\Gamma^*_m=O(m^{-\infty})$ on $X\times X$. From this identity and \eqref{e-gue160905}, we deduce $\Gamma^*_m=S_m+O(m^{-\infty})$ on $X\times X$. Since $S_m$ is self-adjoint, we conclude $\Gamma_m=S_m+O(m^{-\infty})$ on $X\times X$. The theorem follows.
\end{proof}

%From Theorem~\ref{t-gue160904fI}, \eqref{e-gue150626fIII} and \eqref{e-gue150901a}, we deduce

%\begin{thm}\label{t-gue160905}
%Let $b_j(m,x,y)\in C^\infty(X\times X)$, $j=0,1,2,\ldots$, be as in \eqref{e-gue150901a}. Then one has
%\begin{equation}\label{e-gue160905w}
%S_m(x)\sim m^nb_0(m,x,x)+m^{n-1}b_1(m,x,x)+\cdots\ \ \mbox{as $m\To+\infty$}.
%\end{equation}
%\end{thm}

\subsection{The proof of Theorem~\ref{t-gue160605} and Theorem~\ref{t-gue160605I}}\label{s-gue160905}

We will use the same notations as in  Theorem~\ref{t-gue160605} and Theorem~\ref{t-gue160605I}.
%For every $q\in\mathbb N$, put $X_q=\left\{x\in X;\, \mbox{the period of $x$ is $\frac{2\pi}{q}$}%
%\right\}$ and let $p=\min\left\{q\in\mathbb{N};\,
%X_q\neq\emptyset\right\}$. Since $X$ is connected, it is well-known that $X_p$ is an open and dense subset of $X$ (see Duistermaat-Heckman~\cite{Du82}). We denote $X_{{\rm reg\,}}:=X_{p}$. Let $X_{{\rm sing\,}}$ be the complement of $X_{{\rm reg\,}}$. Assume $X=X_{p_1}\bigcup X_{p_2}\bigcup\cdots\bigcup X_{p_t}$, $p=:p_1<p_2<\cdots<p_t$. Put $X_{{\rm sing\,}}=X^1_{{\rm sing\,}}:=\bigcup^t_{j=2}X_{p_j}$, $X^{r}_{{\rm sing\,}}:=\bigcup^t_{j=r+1}X_{p_j}$, $t-1\geq r\geq 2$. Set $X^{t}_{{\rm sing\,}}:=\emptyset$. Fix $x, y\in X$. Let $d(x,y)$ denotes the standard Riemannian distance of $x$ and $y$ with respect to the given Hermitian metric.
Take
\begin{equation}  \label{e-gue160327}
0<\delta<\inf\left\{\frac{\pi}{p_t},\left\vert \frac{2\pi}{p_r}-\frac{2\pi}{
p_{r+1}}\right\vert, r=1,\ldots,t-1\right\}.
\end{equation}
For $x\in X$ and every $r=1,2,\ldots,t$, set
\begin{equation}\label{e-gue160929}
\hat d_\delta(x,X^r_{\mathrm{sing\,}}):=\inf\left\{d(x,e^{-i\theta}x);\,
\delta\leq\theta\leq\frac{2\pi}{p_r}-\delta\right\}.
\end{equation}
Note that for any $0<\delta, \delta_1$ satisfying \eqref{e-gue160327}, $\hat
d_\delta(x,X^r_{\mathrm{sing\,}})$ and $\hat d_{\delta_1}(x,X^r_{\mathrm{sing\,%
}})$ are equivalent. We shall denote $\hat d(x,X^r_{\mathrm{sing\,}}):=\hat
d_\delta(x,X^r_{\mathrm{sing\,}})$. The following is well-known (see Theorem 6.5 in~\cite{CHT})

\begin{thm}\label{t-gue160910c}
There is a constant $C\geq1$ such that
\begin{equation*}
\frac{1}{C}d(x,X^r_{\mathrm{sing\,}})\leq\hat d(x,X^r_{\mathrm{sing\,}})\leq
Cd(x,X^r_{\mathrm{sing\,}}),\ \ \forall x\in X.
\end{equation*}
\end{thm}

Let
\begin{equation}\label{e-gue161009}
\begin{split}
\Gamma^0_m(x,y)&=\frac{1}{2\pi}\sum^N_{j=1}\int^{2\delta}_{-2\delta}H_{j,m}(x,e^{iu}\circ
y)e^{imu}du \\
&=\frac{1}{2\pi}\sum^N_{j=1}\int^{2\delta}_0e^{im\hat
\Psi_{B_j}(x,e^{iu}\circ y)}\hat b_{B_j,m}(x,e^{iu}\circ y)e^{imu}du.
\end{split}
\end{equation}
%Let $\widetilde H_{j,m}(x,y)$ be as in \eqref{e-gue150627f}. We will
%use the same notations as before. From the construction of $\widetilde
%H_{j,m}(x,y)$ and Theorem~\ref{t-gue160902I},
It is easy to check that
\begin{equation}  \label{e-gue160224I}
\begin{split}
&\Gamma^0_m(x,x)\sim
m^{n}b_0(x)+m^{n-1}b_1(x)+m^{n-2}b_2(x)+\cdots\ \ \mbox{in $S^n_{{\rm loc\,}}(1,X)$},
\\
&b_k(x)=\frac{1}{2\pi}\sum^N_{j=1}\chi_j(x)b_{B_j,k}(z,z)\in C^\infty(X),\ \ k=0,1,2,\ldots,
\end{split}
\end{equation}
and
\begin{equation} \label{e-gue160905c}
b_{0}(x)=(2\pi)^{-n-1}\det\dot{\mathcal{R}}(x),
\end{equation}
\begin{equation} \label{e-gue160905cI}
b_{1}(x)=(2\pi)^{-n-1}\det\dot{\mathcal{R}}(x)\Bigr(\frac{1}{4\pi}S^\Theta_{\mathcal{L}}-\frac{1}{8\pi}S_{\mathcal{L}}\Bigr)(x),
\end{equation}
\begin{equation} \label{e-gue160906cII}
\begin{split}
b_{2}(x)&=(2\pi)^{-n-1}\det\dot{\mathcal{R}}(x)\Bigr(\frac{1}{128\pi^2}(S_{\mathcal{L}})^2-\frac{1}{32\pi^2}S_{\mathcal{L}}S^\Theta_{\mathcal{L}}+\frac{1}{32\pi^2}(S^\Theta_{\mathcal{L}})^2
-\frac{1}{32\pi^2}\triangle_{\mathcal{L}}S^\Theta_{\mathcal{L}}
-\frac{1}{8\pi^2}\abs{R^{\det}_\Theta}^2_{\mathcal{L}}\\
&\quad+\frac{1}{8\pi^2}\langle\,{\rm Ric\,}_{\mathcal{L}}\,|\,R^{\det}_\Theta\,\rangle_{\mathcal{L}}+\frac{1}{96\pi^2}\triangle_{\mathcal{L}}S_{\mathcal{L}}-
\frac{1}{24\pi^2}\abs{{\rm Ric\,}_{\mathcal{L}}}^2_{\mathcal{L}}+\frac{1}{96\pi^2}\abs{R^{T^{1,0}X}_{\mathcal{L}}}^2_{\mathcal{L}}\Bigr)(x).
\end{split}
\end{equation}

\begin{thm}\label{t-gue160905c}
With the notations used above, there is an $\varepsilon_0>0$ such that for any $r=1,\ldots,t$ and
any differential operator $P_\ell:C^\infty(X)\rightarrow
C^\infty(X)$ of order $\ell\in\mathbb{N}_0$, there is a  $C>0$ such that
\begin{equation}  \label{e-gue160910}
\begin{split}
&\abs{P_\ell\Bigr(\Gamma_m(x,x)-\sum
\limits^{p_r}_{s=1}e^{\frac{2\pi(s-1)}{p_r}mi}\Gamma^0_m(x,x)\Bigr)}\\
&\leq Cm^{n+\frac{\ell}{2}}e^{-\varepsilon_0md(x,X^r_{{\rm sing\,}})^2},\ \ \forall  x\in X_r.
\end{split}
\end{equation}
\end{thm}

\begin{proof}
%For simplicity, we only give a detailed proof of \eqref{e-gue160910} when $P_\ell$ is the identity map. For general $P_\ell$, the proof is similar.
We first prove \eqref{e-gue160910} when $P_\ell$ is the identity map.
Let $B_j:=(D_j,(z,\theta),\varphi_j)$, $\hat D_j$, $
j=1,\ldots,N$ be as in the beginning of Section~\ref{s-gue160903}. We will
use the same notations as in Section~\ref{s-gue160903}. For each $j$, we
have $D_j=U_j\times]-4\delta,4\delta[\subset\mathbb{C}^n\times
\mathbb{R}$, $\hat D_j=\hat U_j\times]-\frac{\delta}{2},\frac{
\delta}{2}[\subset\mathbb{C}^n\times\mathbb{R}$, $\delta>0$, $U_j=\left\{z\in\mathbb{C}^n;\, \left\vert
z\right\vert<\gamma_j\right\}$, $\hat U_j=\left\{z\in\mathbb{C}^n;\,
\left\vert z\right\vert<\frac{\gamma_j}{2}\right\}$.  Recall that $\delta>0$
satisfies \eqref{e-gue160327}. We may assume that $4\left\vert
\delta\right\vert<\frac{2\pi}{p}$. Recall that $X=\hat
D_1\bigcup\cdots\bigcup\hat D_N$. It is straightforward to see that there is
an $\hat\varepsilon_0>0$ such that for each $j=1,\ldots,N$, we have
\begin{equation}  \label{e-gue160215}
\begin{split}
&\hat\varepsilon_0d((z_1,\theta_1),(z_2,\theta_1))\leq\left\vert
z_1-z_2\right\vert\leq\frac{1}{\hat\varepsilon_0}d((z_1,\theta_1),(z_2,%
\theta_1)), \forall (z_1,\theta_1),(z_2,\theta_1)\in\hat D_j, \\
&\hat\varepsilon_0d((z_1,\theta_1),(z_2,\theta_1))^2\leq\Psi_{B_j}(z_1,z_2)\leq%
\frac{1}{\hat\varepsilon_0}d((z_1,\theta_1),(z_2,\theta_1))^2, \forall
(z_1,\theta_1),(z_2,\theta_1)\in\hat D_j,
\end{split}
\end{equation}
where $\Psi_{B_j}(z,w)$ is as in \eqref{e-gue160902II}.

Fix $x_0\in X_{p_r}$. Fix $j=1,2,\ldots,N$. Assume that $x_0\in\hat D_j$ and suppose
that $x_0=(z_0,\theta_0)\in\hat D_j$. We claim
\begin{equation}  \label{e-gue160327I}
\begin{split}
&\mbox{if $e^{i\theta}\circ x_0=(\Td z,\Td\eta)\in\hat D_j$, for some
$\theta\in[2\delta,\frac{2\pi}{p_r}-2\delta]$}, \\
&\mbox{then $\abs{z-\Td z}\geq\hat\varepsilon_0\hat d(x_0,X^r_{{\rm sing\,}})$}
\end{split}
\end{equation}
where $\hat d(x_0,X^r_{{\rm sing\,}})=\hat d_\delta(x_0,X^r_{{\rm sing\,}})$ is as in \eqref{e-gue160929}.

\begin{proof}[proof of the claim]
We have
\begin{equation}  \label{e-gue160909}
\begin{split}
\abs{z_0-\Td z}&=\abs{(z_0,\theta_0)-(\Td z,\theta_0)}=\abs{(z_0,\theta_0)-e^{-i\Td\eta+i\theta_0}(\Td z,\Td\eta)}\\
&=\abs{x_0-e^{i(-\Td\eta+\theta_0)}\circ e^{i\theta}\circ x_0}.
\end{split}
\end{equation}
Note that $\abs{-\Td\eta+\theta_0}\leq\delta$ holds. From this observation, \eqref{e-gue160909} and \eqref{e-gue160215} we have
\begin{equation}  \label{e-gue160327II}
\begin{split}
\left\vert \widetilde
z-z_0\right\vert&\geq\hat\varepsilon_0\inf\left\{d(e^{iu}\circ
e^{i\theta}\circ x_0,x_0);\, \left\vert u\right\vert\leq\delta\right\} \\
&\geq\hat\varepsilon_0\inf\left\{d(e^{i\theta}\circ x_0,x_0);\,
\delta\leq\theta\leq\frac{2\pi}{p_r}-\delta\right\}=\hat\varepsilon_0\hat
d(x_0,X^r_{\mathrm{sing\,}}).
\end{split}
\end{equation}
The claim follows.
\end{proof}

From \eqref{e-gue150626fIII}, we have
\begin{equation}  \label{e-gue160909I}
\begin{split}
&\Gamma_m(x_0,x_0) \\
&=\frac{1}{2\pi}\sum^N_{j=1}\int^{2\pi}_{0}e^{im\hat
\Psi_{B_j}(x_0,e^{iu}\circ x_0)}\hat b_{B_j,m}(x_0,e^{iu}\circ x_0)e^{imu}du\\
&=\frac{1}{2\pi}\sum\limits^{p_r}_{s=1}e^{\frac{2\pi(s-1)}{p_r}mi}\sum^N_{j=1}\int^{\frac{2\pi}{p_r}}_{0}e^{im\hat
\Psi_{B_j}(x_0,e^{iu}\circ x_0)}\hat b_{B_j,m}(x_0,e^{iu}\circ x_0)e^{imu}du\\
&=\frac{1}{2\pi}\sum\limits^{p_r}_{s=1}e^{\frac{2\pi(s-1)}{p_r}mi}\sum^N_{j=1}\int_{u\in[2\delta,\frac{2\pi}{p_r}-2\delta]}e^{im\hat
\Psi_{B_j}(x_0,e^{iu}\circ x_0)}\hat b_{B_j,m}(x_0,e^{iu}\circ x_0)e^{imu}du\\
&\quad+\frac{1}{2\pi}\sum\limits^{p_r}_{s=1}e^{\frac{2\pi(s-1)}{p_r}mi}\sum^N_{j=1}\int^{2\delta}_{-2\delta}e^{im\hat
\Psi_{B_j}(x_0,e^{iu}\circ x_0)}\hat b_{B_j,m}(x_0,e^{iu}\circ x_0)e^{imu}du\\
&=\frac{1}{2\pi}\sum\limits^{p_r}_{s=1}e^{\frac{2\pi(s-1)}{p_r}mi}\sum^N_{j=1}\int_{u\in[2\delta,\frac{2\pi}{p_r}-2\delta]}e^{im\hat
\Psi_{B_j}(x_0,e^{iu}\circ x_0)}\hat b_{B_j,m}(x_0,e^{iu}\circ x_0)e^{imu}du\\
&\quad+\sum
\limits^{p_r}_{s=1}e^{\frac{2\pi(s-1)}{p_r}mi}\Gamma^0_m(x_0,x_0).
\end{split}
\end{equation}
From \eqref{e-gue160215}, \eqref{e-gue160327I} and Theorem~\ref{t-gue160910c}, there are $\varepsilon_0>0
$ and $C_0>0$ independent of $k$ and $x_0$ such that
\begin{equation}  \label{e-gue160125V}
\begin{split}
&\abs{\frac{1}{2\pi}\sum\limits^{p_r}_{s=1}e^{\frac{2\pi(s-1)}{p_r}mi}\sum^N_{j=1}\int_{u\in[2\delta,\frac{2\pi}{p_r}-2\delta]}e^{im\hat
\Psi_{B_j}(x_0,e^{iu}\circ x_0)}\hat b_{B_j,m}(x_0,e^{iu}\circ x_0)e^{imu}du}\\
&\leq C_0m^ne^{
-\varepsilon_0md(x_0,X^r_{{\rm sing\,}})^2},\ \ \forall m>0.
\end{split}
\end{equation}
From \eqref{e-gue160125V} and \eqref{e-gue160909I}, we get the theorem when $P_\ell$ is the identity map.

Now, let's consider general $P_\ell$.  For simplicity, we only prove \eqref{e-gue160910} when $\ell=1$. For general $P_\ell$, the proof is similar. Fix $x_0\in X_{p_r}$. Fix $j=1,2,\ldots,N$. Assume that $x_0\in\hat D_j$ and suppose that $x_0=(z_0,\theta_0)\in\hat D_j$. Let $x=(z,\theta)$ be the canonical coordinates of $X$ on $\hat D_j$, $z_k=x_{2k-1}+ix_{2k}$, $j=1,\ldots,n$, $x_{2n}=\theta$.
%It is clear that there is a constant $C>0$ independent of $x_0$ such that
%\begin{equation}\label{e-gue161010}
%\begin{split}
%&\abs{P_\ell\Bigr(\Gamma_m(x,x)-\sum
%\limits^{p_r}_{s=1}e^{\frac{2\pi(s-1)}{p_r}mi}\Gamma^0_m(x,x)\Bigr)|_{x=x_0}}\\
%&\leq C\sum^{2n+1}_{k=1}\abs{\frac{\pr}{\pr x_k}\Bigr(\Gamma_m(x,x)-\sum
%\limits^{p_r}_{s=1}e^{\frac{2\pi(s-1)}{p_r}mi}\Gamma^0_m(x,x)\Bigr)|_{x=x_0}}.
%\end{split}
%\end{equation}
For every $k=1,2,\ldots,2n$, we have
%From the fact that $\Psi_{B_j}((z,\theta),(z,\eta))=0$, it is not difficult to check that
\begin{equation}\label{e-gue161010I}
\frac{\pr}{\pr x_k}\Bigr(\int^{2\pi}_{0}e^{im\hat
\Psi_{B_j}(x,e^{iu}\circ x)}\hat b_{B_j,m}(x,e^{iu}\circ x)e^{imu}du\Bigr)=\int^{2\pi}_{0}e^{im\hat
\Psi_{B_j}(x,e^{iu}\circ x)}\Td b^k_{B_j,m}(x,e^{iu}\circ x)e^{imu}du,
\end{equation}
where
\begin{equation}\label{e-gue161010II}
\Td b^k_{B_j,m}(x,e^{iu}\circ x)=im\Bigr(\frac{\pr}{\pr x_k}\hat
\Psi_{B_j}(x,e^{iu}\circ x)\Bigr)\hat b_{B_j,m}(x,e^{iu}\circ x)+\frac{\pr}{\pr x_k}\hat b_{B_j,m}(x,e^{iu}\circ x).
\end{equation}
We can repeat the procedure \eqref{e-gue160909I} with minor change and get that
\begin{equation}  \label{e-gue161010III}
\begin{split}
&\frac{\pr}{\pr x_k}\Bigr(\int^{2\pi}_{0}e^{im\hat
\Psi_{B_j}(x,e^{iu}\circ x)}\hat b_{B_j,m}(x,e^{iu}\circ x)e^{imu}du\Bigr)|_{x=x_0} \\
&=\sum\limits^{p_r}_{s=1}e^{\frac{2\pi(s-1)}{p_r}mi}\int_{u\in[2\delta,\frac{2\pi}{p_r}-2\delta]}e^{im\hat
\Psi_{B_j}(x_0,e^{iu}\circ x_0)}\Td b^k_{B_j,m}(x_0,e^{iu}\circ x_0)e^{imu}du\\
&\quad+\sum
\limits^{p_r}_{s=1}e^{\frac{2\pi(s-1)}{p_r}mi}\frac{\pr}{\pr x_k}\Bigr(\int^{2\delta}_{-2\delta}e^{im\hat
\Psi_{B_j}(x,e^{iu}\circ x)}\hat b_{B_j,m}(x,e^{iu}\circ x)e^{imu}du\Bigr)|_{x=x_0}.
\end{split}
\end{equation}
From the fact that $\Psi_{B_j}((z,\theta),(z,\eta))=0$ and $\frac{\pr}{\pr x_k}\hat
\Psi_{B_j}(x,e^{iu}\circ x)=O(\abs{z-z'})$, where $x=(z,\theta)\in\hat D_j$, $e^{iu}\circ x=(z',\theta')\in\hat D_j$, and by using the claim \eqref{e-gue160327I}, we conclude that for every $u\in[2\delta,\frac{2\pi}{p_r}-2\delta]$, we have
\begin{equation}\label{e-gue161010ry}
\abs{\Td b^k_{B_j,m}(x_0,e^{iu}\circ x_0)}\leq Cm^{n}\Bigr(md(x_0,X^r_{{\rm sing\,}})+1\Bigr),
\end{equation}
where $C>0$ is a constant independent of $x_0$ and $u$. From \eqref{e-gue161010III} and \eqref{e-gue161010ry}, we deduce that
\begin{equation}  \label{e-gue161010ryI}
\begin{split}
&\Bigr|\frac{\pr}{\pr x_k}\Bigr(\int^{2\pi}_{0}e^{im\hat
\Psi_{B_j}(x,e^{iu}\circ x)}\hat b_{B_j,m}(x,e^{iu}\circ x)e^{imu}du\Bigr)|_{x=x_0}\\
&\quad-\sum
\limits^{p_r}_{s=1}e^{\frac{2\pi(s-1)}{p_r}mi}\frac{\pr}{\pr x_k}\Bigr(\int^{2\delta}_{-2\delta}e^{im\hat
\Psi_{B_j}(x,e^{iu}\circ x)}\hat b_{B_j,m}(x,e^{iu}\circ x)e^{imu}du\Bigr)|_{x=x_0}\Bigr|\\
&\leq C_1m^{n+\frac{1}{2}}e^{-\varepsilon_0md(x_0,X^r_{{\rm sing\,}})^2},
\end{split}
\end{equation}
where $C_1>0$ and $\varepsilon_0>0$ are constants independent of $x_0$ and $\varepsilon_0>0$ is independent of the derivative $\frac{\pr}{\pr x_k}$. On $\hat D_j$, write
\[P_\ell=\sum^{2n}_{k=1}a_k(x)\frac{\pr}{\pr x_k}+a_{2n+1}(x)T,\ \ a_k\in C^\infty(\hat D_j),\ \ k=1,2,\ldots,2n+1.\]
Let $P^0_\ell=\sum^{2n}_{k=1}a_k(x)\frac{\pr}{\pr x_k}$. From \eqref{e-gue161010ryI}, we deduce that
\begin{equation}  \label{e-gue161010ryII}
\begin{split}
&\Bigr|P^0_\ell\Bigr(\int^{2\pi}_{0}e^{im\hat
\Psi_{B_j}(x,e^{iu}\circ x)}\hat b_{B_j,m}(x,e^{iu}\circ x)e^{imu}du\\
&\quad-\sum
\limits^{p_r}_{s=1}e^{\frac{2\pi(s-1)}{p_r}mi}\int^{2\delta}_{-2\delta}e^{im\hat
\Psi_{B_j}(x,e^{iu}\circ x)}\hat b_{B_j,m}(x,e^{iu}\circ x)e^{imu}du\Bigr)|_{x=x_0}\Bigr|\\
&\leq C_2\Bigr(\sum^{2n}_{k=1}\Bigr|\frac{\pr}{\pr x_k}\Bigr(\int^{2\pi}_{0}e^{im\hat
\Psi_{B_j}(x,e^{iu}\circ x)}\hat b_{B_j,m}(x,e^{iu}\circ x)e^{imu}du\Bigr)\Bigr|_{x=x_0}\\
&\quad-\sum
\limits^{p_r}_{s=1}e^{\frac{2\pi(s-1)}{p_r}mi}\frac{\pr}{\pr x_k}\Bigr(\int^{2\delta}_{-2\delta}e^{im\hat
\Psi_{B_j}(x,e^{iu}\circ x)}\hat b_{B_j,m}(x,e^{iu}\circ x)e^{imu}du\Bigr)|_{x=x_0}\Bigr|\Bigr)\\
&\leq C_3m^{n+\frac{1}{2}}e^{-\varepsilon_1md(x_0,X^r_{{\rm sing\,}})^2},
\end{split}
\end{equation}
where $C_2>0$, $C_3>0$ and $\varepsilon_1>0$ are constants independent of $x_0$ and $\varepsilon_1>0$ is independent of $P^0_\ell$.  From \eqref{e-gue161010ryII}, we obtain
\begin{equation}  \label{e-gue161010ryIII}
\begin{split}
&\Bigr|P_\ell\Bigr(\Gamma_m(x,x)-\sum
\limits^{p_r}_{s=1}e^{\frac{2\pi(s-1)}{p_r}mi}\Gamma^0(x,x)\Bigr)|_{x=x_0}\Bigr|\\
&\leq\hat C_0\Bigr(m^{n+\frac{1}{2}}e^{-\varepsilon_2md(x_0,X^r_{{\rm sing\,}})^2}+\Bigr|T\Bigr(\Gamma_m(x,x)-\sum
\limits^{p_r}_{s=1}e^{\frac{2\pi(s-1)}{p_r}mi}\Gamma^0(x,x)\Bigr)|_{x=x_0}\Bigr|\Bigr),
\end{split}
\end{equation}
where $\hat C_0>0$ and $\varepsilon_2>0$ are constants independent of $x_0$ and $\varepsilon_2>0$ is independent of $P_\ell$. Let us estimate $\Bigr|T\Bigr(\Gamma_m(x,x)-\sum
\limits^{p_r}_{s=1}e^{\frac{2\pi(s-1)}{p_r}mi}\Gamma^0(x,x)\Bigr)|_{x=x_0}\Bigr|\Bigr)$. We can repeat the procedure \eqref{e-gue160909I} with minor change and deduce that\begin{equation}  \label{e-gue161010rya}
\begin{split}
\Bigr|T\Bigr(\Gamma_m(x,x)-\sum
\limits^{p_r}_{s=1}e^{\frac{2\pi(s-1)}{p_r}mi}\Gamma^0(x,x)\Bigr)|_{x=x_0}\Bigr|\leq\hat C_1m^{n+1}e^{-\varepsilon md(x_0,X^r_{{\rm sing\,}})^2},
\end{split}
\end{equation}
where $\hat C_1>0$ and $\varepsilon>0$ are constants independent of $x_0$  and $\varepsilon>0$ is independent of  the derivative $T$. Since $\Gamma_m(x,x)=S_m(x)+O(m^{-\infty})$ and $TS_m(x)=0$, we deduce that $T\Gamma_m(x,x)=O(m^{-\infty})$. Moreover, from \eqref{e-gue160224I}, we find $\abs{T\Gamma^0_m(x,x)}\leq\hat Cm^n$, for every $x\in X$, where $\hat C>0$ is a constant. From this observation, we conclude
\begin{equation}  \label{e-gue161012}
\begin{split}
\Bigr|T\Bigr(\Gamma_m(x,x)-\sum
\limits^{p_r}_{s=1}e^{\frac{2\pi(s-1)}{p_r}mi}\Gamma^0(x,x)\Bigr)|_{x=x_0}\Bigr|\leq\Td C_1m^n,
\end{split}
\end{equation}
where $\Td C_1>0$ is a constant independent of $x_0$. From \eqref{e-gue161010rya} and \eqref{e-gue161012}, we get
\begin{equation}  \label{e-gue161012I}
\begin{split}
\Bigr|T\Bigr(\Gamma_m(x,x)-\sum
\limits^{p_r}_{s=1}e^{\frac{2\pi(s-1)}{p_r}mi}\Gamma^0(x,x)\Bigr)|_{x=x_0}\Bigr|\leq\Td C_2m^{n+\frac{1}{2}}e^{-\Td\varepsilon md(x_0,X^r_{{\rm sing\,}})^2},
\end{split}
\end{equation}
where $\Td C_2>0$ and $\Td\varepsilon>0$ are constants independent of $x_0$ and $\Td\varepsilon>0$ is independent of the derivative $T$ . From \eqref{e-gue161010ryIII} and \eqref{e-gue161012I}, we get \eqref{e-gue160910} when $\ell=1$. For general $P_\ell$, the proof is similar.
\end{proof}

From Theorem~\ref{t-gue160904fI}, Theorem~\ref{t-gue160905c}, \eqref{e-gue160224I}, \eqref{e-gue160905c}, \eqref{e-gue160905cI} and \eqref{e-gue160906cII}, we get Theorem~\ref{t-gue160605} and Theorem~\ref{t-gue160605I}.
\begin{center}
{\bf Acknowledgement}
\end{center}

 The authors thank the referees for careful reading the manuscript and many useful suggestion which improve the presentation of this work.
\bibliographystyle{plain}

\end{document}